\documentclass{amsart}

\usepackage{enumerate}
\usepackage{amsmath}
\usepackage{epsfig}
\usepackage{graphics}
\usepackage[normalem]{ulem}
\usepackage{color}

\usepackage{subcaption}

\input xy 
\xyoption{all}

\def\C{\mathbb C}

\def\bcases{\begin{cases}}

\def\ecases{\end{cases}}

\newcommand{\D}{\mathbb D}

\newcommand{\loc}{{\operatorname{loc}}}
\newcommand{\N}{\mathbb N}

\newcommand{\R}{\mathbb R}

\def\sm{\setminus}

\newcommand{\UU}{\mathcal{U}}

\newtheorem{thm}{Theorem}[section]
\newtheorem{main theorem}[thm]{Main Theorem}
\newtheorem{corollary}[thm]{Corollary}

\newtheorem{lemma}[thm]{Lemma}
\newtheorem{prop}[thm]{Proposition}

\newtheorem{problem 1}{Problem 1}
\newtheorem{problem 2}{Problem 2}
\newtheorem{problem 3}{Problem 3}
\theoremstyle{definition}

\newtheorem{defn}[thm]{Definition}
\newtheorem{remark}[thm]{Remark}

\newcommand{\FF}{\mathcal F}
\newcommand{\bea}{\begin{eqnarray*}}
\newcommand{\eea}{\end{eqnarray*}}

\newcommand{\note}[1]{\medskip {\bf  #1}\medskip}

\newcommand{\be}{\begin{equation}}
\newcommand{\ee}{\end{equation}}

\newcommand{\De}{\Delta}
\newcommand{\Si}{\Sigma}
\newcommand{\Om}{\Omega}
\newcommand{\NN}{{\mathcal{N}}}
\newcommand{\VV}{{\mathcal{V}}}
\newcommand{\HH}{{\mathcal{H}}}

\newcommand{\bL}{{\mathbb{L}}}

\newcommand{\epsi}{{\varepsilon}}
\newcommand{\de} {{\delta}}
\newcommand{\la}{{\lambda}}

\newcommand{\dist}{{\operatorname{dist}}}
\newcommand{\isom}{\approx}
\newcommand{\ra}{\rightarrow}
\newcommand{\di}{\partial}
\newcommand {\ssk} {\smallskip}
\newcommand {\msk} {\medskip}
\newcommand{\nin} {\noindent}
\newcommand{\crit} {{\mathrm{crit}}}

\newcommand{\Jac} {\operatorname {Jac}}

\newcommand{\length}{\operatorname{length}}
\newcommand{\diam}{\mathrm{diam}}

\newcommand{\comm}[1]{}

\def\note{\marginpar}

\begin{document}

\title[Partially hyperbolic H\'enon maps]
{Structure of partially hyperbolic H\'enon maps}

\author{Misha Lyubich and Han Peters}

\today

\begin{abstract}
We consider the structure of substantially dissipative complex H\'enon maps admitting a dominated splitting on the Julia set. The dominated splitting assumption corresponds to the one-dimensional assumption that there are no critical points on the Julia set. Indeed, we prove the corresponding description of the Fatou set, namely that it consists of only finitely many components, each either attracting or parabolic periodic. In particular there are no rotation domains, and no wandering components. Moreover, we show that $J = J^\star$ and the dynamics on $J$ is hyperbolic away from parabolic cycles.
\end{abstract}

\maketitle

\section{Introduction}

Complex H\'enon maps are
polynomial automorphisms of $\C^2$ with non-trivial dynamical behavior,
$$
     f: (x, y) \mapsto (p(x) - by, x), \quad {\mathrm{where}}\ \deg p \geq 2 , \ b= \Jac f \not=0.
$$
For a small Jacobian $b$, it can be viewed as a perturbation of the one-dimensional polynomial $p: \C\ra \C$.
Though some initial aspects of the 2D theory resembles the 1D theory,
quite quickly it becomes much more difficult, exhibiting various new phenomena.

Dynamics of 1D polynomials on the Fatou set is fully understood, due to the  classical work of Fatou, Julia and Siegel,
supplemented with  Sullivan's No Wandering Domains Theorem from the early '80s  \cite{Sullivan} .
This direction of research  for H\'enon maps was initiated   by Bedford and Smillie in the early '90's.
In particular, they gave a description of the dynamics on ``recurrent''  periodic Fatou components \cite{BS1991b}.
The ``non-recurrent'' case was   recently treated by the authors \cite{LP2014},
under  an  assumption that the H\'enon map  is ``substantially dissipative'', i.e.
$$
\left| \mathrm \Jac f\right| < \frac{1}{( \deg p)^2 }.
$$
It completed the classification of periodic Fatou components in this setting:
Any such component is either an attracting or parabolic basin, or a rotation domain,
which is analogous  to the one-dimensional classification.
\footnote{One fine pending issue  still unresolved for H\'enon maps is whether Herman rings can exist.}

The situation with the problem of wandering components is more complicated.
In fact, wandering Fatou components can exist for  polynomial endomorphisms $g: \C^2\ra \C^2$,
as  was recently demonstrated in \cite{ABDPR}.
It is probable that wandering components can exist for H\'enon maps as well,
but one can hope that  ``generically''  they do not.

It is quite clear that Sullivan's proof of the Non-Wandering Domains Theorem,
based upon quasiconformal deformations machinery,  is not generalizable to higher dimensions.
At the same time,  for various special classes of 1D polynomial maps,
one can give a direct geometric argument that has a chance to be generalized to the 2D setting.
The simplest class of this kind comprises hyperbolic polynomials,
for which absence of wandering component was known classically.
The H\'enon counterpart of this result was established  by Bedford and Smillie in the '90s,
resulting in a complete description of the dynamics on the Fatou set for this class  \cite{BS1991a}:
a hyperbolic H\'enon map has only finitely many Fatou components, each of which is an attracting basin.
Moreover, in this case, the Julia set $J$ is the closure of saddles: $J= J^*$.

\msk
Until now, hyperbolic maps remained the only class of H\'enon maps for which these problems were settled down.
In this paper, we are making one step further, resolving these problems for substantially dissipative H\'enon maps that admit a ``dominated splitting'' over the Julia set:

\begin{thm}\label{thm:main}
Let $f: \C^2 \ra \C^2$ be a  substantially dissipative  H\'enon map that admits a  dominated splitting  over the Julia set $J$.
Then:

\ssk\nin $\bullet$
 $f$ does not have wandering Fatou components;

\ssk\nin $\bullet$
$f$ has only finitely many periodic  Fatou components,
  each being  either an attracting  or a parabolic basin;

\ssk\nin $\bullet$ $J=J^*$, i.e., the Julia set is the closure of saddles.
\end{thm}

H\'enon maps with dominated splitting  are 2D counterparts of  1D polynomial maps without critical points on the Julia set.
Our initial observation was that for such a polynomial, the No Wandering Domains Theorem can be proven by means
of Ma\~n\'e's techniques (refined in \cite{CJY} and \cite{STL})  treating  maps with non-recurrent critical points.
However, an adaptation of these techniques to the H\'enon setting is
not straightforward:
in particular, it required to impose  an assumption of  substantial dissipativity,
to develop an appropriate version of the $\lambda$-lemma,  and to bound  the
iterated  degrees of wandering components.

 The main work in this paper is to show that, away from the parabolic cycles,
$f$ is expanding in the horizontal direction.
(In particular,  if there are no parabolic cycles then $f$ is hyperbolic.)
The non-existence of wandering Fatou components and periodic  rotation
domains  follows easily,
and it also follows that $J = J^*$.
Moreover, we  show that  if there are parabolic cycles,
 then $f$ lies on the boundary of the hyperbolicity locus,
at least when viewed in the parameter space of H\'enon-like maps.

Non-hyperbolic complex H\'enon maps admitting a dominated splitting
have been constructed by Radu and Tanase \cite{RT}.
These examples are perturbations of 1D parabolic polynomials.

\msk
The structure of the paper is as follows.
In section (2) we will review   Ma\~n\'e's Theorem
analysing the dynamics of   1D polynomials without recurrent critical points (except possible superattracting cycles).
We give a detailed proof that follows ideas from a paper of Shishikura and Tan Lei \cite{STL}.
However, we have chosen to present an argument that will be a closest possible  model
to the two-dimensional proof we give later.
This means that the one-dimensional argument is not the most efficient. For instance, naturally  we do not assume the non-existence of wandering components.

In section (3) we recall H\'enon maps and the substantial dissipativity condition, and in section (4) we define the dominated splitting and make some elementary observations. The dominated splitting on $J$ induces a lamination on $J^+\cap \Delta^2_R$ by vertical disks. In section (5) this lamination is extended to a neighborhood of $J^+$, introducing the artificial vertical lamination. While this lamination is not invariant, it plays an important role in our proofs.

In wandering Fatou components we can consider both the artificial and the dynamical lamination given by strong stable manifolds. These two laminations may not agree on orbits of wandering domains that leave the region of dominated splitting, leading to interpretations of degree, discussed in sections (6) and (7).

In section (8) we make the final preparations and in section (9) we prove the main technical result, Proposition \ref{Prop51}. In section (10) we prove the consequences of this proposition, including Theorem \ref{thm:main} above.

\msk
In conclusion, let us note that
dominated splitting is an important classical  notion  going back to the works of Pliss and  Ma\~n\'e
from the '70s. Dynamics of real surface diffeomorphisms with dominated splitting
was described by Pujals and Samborino \cite{PS}.
 This result inspired  our  work.

\msk {\bf Acknowledgment.}
It is our pleasure to  thank Enrique Pujals,  Vladlen Timorin,
Remus Radu and Raluca Tanase
for very interesting discussions of the  dominated splitting theme.
This work has been partially supported by the NSF, NSERC,
and the Simons Foundation.

\section{The one-dimensional argument}

Let $f: \mathbb C \rightarrow \mathbb C$ be a polynomial, and assume that there are no critical points on $J$, the Julia set of $f$. We let $\Omega$ be a backward invariant open neighborhood of $J \setminus \{parabolics\}$, constructed by removing closed forward invariant sets from a finite number of (pre-) periodic Fatou components. We will assume that $\Omega$ is arbitrarily thin, i.e. contained in an arbitrarily small neighborhood of the union of $J \setminus \{parabolics\}$ with the wandering Fatou components.

The wandering Fatou components can contain only a finite number of critical points $x_1, \ldots , x_\nu$, having respective local degrees $d_1, \ldots , d_\nu$. We denote
$$
\mathrm{deg}_{\mathrm{crit}} = \prod_{s = 1}^\nu d_s.
$$
The constant $\mathrm{deg}_{\mathrm{crit}}$ functions as a maximal local degree on the wandering components for all iterates, i.e. if $V^0, V^1, \ldots , V^n$ is a orbit of open connected sets, each $V^n$ is contained in a wandering domain and has sufficiently small Euclidean diameter, then $f^n: V^0 \rightarrow V^n$ has degree at most $\mathrm{deg}_{\mathrm{crit}}$.

We will use the following shorthand notation. For a connected set $V$, we denote by $V^{-j}$ a connected component of $f^{-j}(V)$. If $V \subset W$ then we will always assume that $V^{-j} \subset W^{-j}$. When working with both $V^{-i}$ and $V^{-j}$, for $j > i$, we will assume that $f^{j-i}(V^{-j}) = V^{-i}$, i.e. that $V^{-i}$ and $V^{-j}$ are contained in the same backward orbit.

We will show that there exist an integer $M \in \mathbb N$ such that the following holds whenever $\Omega$ is sufficiently thin.

\begin{prop}\label{1D-prop1}
Let $z \in \Omega$ and let $r > 0$ be such that $D_{M \cdot r}(z) \subset \Omega$. Then for every $j \in \mathbb N$ we have that
$$
\begin{aligned}
\mathrm{\bf deg}(j): & \; \; \; \; \;\mathrm{deg} \left(f^j: D^{-j}_r(z) \rightarrow D_r(z)\right) \le \mathrm{deg}_{\mathrm{crit}},\\
\mathrm{\bf diam}(j): & \; \; \; \; \; \mathrm{diam}_{\Omega} D^{-j}_r(z) \le N_0(2K) \cdot C(\frac{1}{2}, \mathrm{deg}_{\mathrm{crit}}).
\end{aligned}
$$
\end{prop}

The constants $C(\cdot, \cdot)$ and $N_0(\cdot)$  will be introduced in
Lemmas \ref{basiclemma} and \ref{1D-lemma2} below.

\medskip

Our proof will closely resemble the proof of a theorem of Ma\~n\'e,
presented in \cite{STL}. In fact, readers familiar with this reference
will likely find our proof needlessly complicated. The reason for
these complications is that the proof given here will  model 
the forthcoming  $2$-dimensional proof. In particular, we will \emph{not} use that there are no wandering domains. In fact, the non-existence of wandering domains follows, in our setting, from the above Proposition. The a priori possibility of wandering domains makes the proof significantly more involved.

\medskip

The fact that Fatou components of one-dimensional polynomials are
simply connected is quite useful when dealing with degrees. It follows
that if a Fatou component $U$ does not contain critical points, then
$f: U \rightarrow f(U)$ is univalent. This is another fact that we
will not be able to use in higher dimensions, so we will not use it
here either. In this respect the setting is more analogous to the
iteration of rational functions, where Fatou components may not be
simply connected. An elementary proof however shows that for every
wandering component $U$ there exists an $N\in \mathbb N$ such that
$f^n(U)$ is simply connected for $n \ge N$, a result known as Baker's
Lemma, see for example \cite{Zakeri}. Thus, if such $f^n(U)$ does not
contain critical points, then $f: f^n(U) \rightarrow f^{n+1}(U)$ is
univalent.
We can use the fact that $f: f^n(U) \rightarrow f^{n+1}(U)$ is univalent for $n$ large enough,
 as we will prove the corresponding 
statement for H\'enon maps. Note that nowhere else in the one-dimensional argument will we use simple connectivity to conclude univalence.

\medskip

A third difference with the argument in \cite{STL} concerns the induction procedure. Instead of applying the induction hypothesis to $f^n: D^{-n}_r(z) \rightarrow D_r(z)$, and then mapping backward one more step with $f: D^{-n-1}_r(z)\rightarrow D^{-n}_r(z)$, we will first apply one iterate $f: D^{-1}_r(z) \rightarrow D_r(z)$, cover $D^{-1}_r(z)$ with smaller disks $D_{r_k}(z_k)$, and applying the induction hypothesis to each $f^n: D^{-n}_{r_k}(z_k)\rightarrow D_{r_k}(z_k)$. The reason for this will become apparent when the proof is discussed in the H\'enon setting.

\medskip

\subsection{Preliminaries}
The following basic lemma  (see e.g.,  \cite{LM1997,STL}) will be used repeatedly:

\begin{lemma}\label{basiclemma}
Let $d \in \mathbb N$ and $r> 0$. Then there exists a constant $C(r,d) > 0$ such that for every proper holomorphic map $f: \mathbb D \rightarrow \mathbb D$ of degree at most $d$, every connected component of $f^{-1}D_r(0)$ has hyperbolic diameter at most $C(r,d)$. Moreover, $C(r,d) \rightarrow 0$ as $r \rightarrow 0$.
\end{lemma}

The first step in the proof is the construction of the backwards domain $\Omega$ where the argument of \ref{1D-prop1} will take place.

\begin{lemma}\label{1D:Omega}
Given any $\epsilon>0$, there exists a backward invariant domain $\Omega$ contained in the $\epsilon$-neighborhood of
$$
\left(J \setminus \{\text{parabolic cycles}\}\right) \cup \left( \bigcup \text{wandering domains}\right),
$$
which further has the property that for every $z \in J$ there exist points $w \notin \Omega$ with $|z-w|< \epsilon$.
\end{lemma}
\begin{proof}
We are done if we can remove sufficiently large forward invariant subsets from a finite number of (pre-) periodic Fatou components. By the classification of periodic Fatou components those periodic components are either attracting basins, parabolic basins or Siegel domains.  In an immediate basin of an attracting periodic cycle we can construct an arbitrarily large forward invariant compact subset. In a cycle of Siegel domains we can find an arbitrarily large completely invariant compact subset. Finally, in a cycle of parabolic domains we can find an arbitrarily large forward invariant compact subset that intersects $J$ only in parabolic fixed points. By taking the union of sufficiently large preimages of these subsets we obtain a forward invariant compact subset $K$ disjoint from $J \setminus \{\text{parabolic cycles}\}$, for which $\Omega = \mathbb C \setminus K$ satisfies the conditions required in the lemma.
\end{proof}

The domain $\Omega$ will later be fixed for a constant $\epsilon>0$
chosen sufficiently small.
In particular, we may assume that the only critical points in $\Omega$ lie in wandering Fatou components.

\begin{lemma}\label{1D-lemma2}
For each $t > 1$ there exists an integer $N_0 = N_0(t)$ such that for every disk $D_r(z) \subset D_{t\cdot r}(z) \subset \Omega$ we can cover $D^{-1}_r(z)$ with at most $N_0$ disks $D_{r_k}(z_k)$ satisfying
$$
D_{2t\cdot r_k}(z_k) \subset \Omega.
$$
If $D_{t \cdot r}(z)$ is contained in a wandering domain $U^0$, then the disks $D_{t \cdot r_k}(z_k)$ can be chosen so that
$$
D_{2t\cdot r_k}(z_k) \subset U^{-1}.
$$
In all other cases the disks $D_{t \cdot r_k}(z_k)$ can be chosen so that
$$
D_{2t\cdot r_k}(z_k) \subset D_{t\cdot r}^{-1}(z).
$$
\end{lemma}
\begin{proof}
If $D_r(z)$ is sufficiently small and close to a critical point, it must be contained in a wandering Fatou component $U$. The statement follows immediately.

For sufficiently small disks bounded away from critical values the existence of a uniform bound is clear, as the map $f: D_{t\cdot r}^{-1}(z) \rightarrow D_{t\cdot r}(z)$ is close to linear. Therefore it is sufficient to consider disks $D_r(z) \subset D_{t\cdot r}(z)$ of radius bounded away from zero. This is a compact family of disks, hence the existence of a uniform $N_0$ is immediate.
\end{proof}

We note that while $N_0(t)$ will play a similar role as the constant
$N_0$ in \cite{STL}, its definition differs as the disks
$D_{r_k}(z_k)$ cover a preimage $D^{-1}_r(z)$ instead of the original
$D_r(z)$.
In particular,  the constant $N_0$ from \cite{STL} is a universal, while the constants $N_0(t)$ introduced here depend on $f$.

\begin{remark}
Since the disks $D_{2t\cdot r_k}(z)$ are contained in either a wandering domain $U^{-1}$ or in the preimage $D_{t\cdot k}^{-1}(z)$, it follows that the constant $N_0(t)$ does not depend on $\Omega$. To be more precise, when $\Omega$ is made smaller, the constant $N_0(t)$ does not need to be changed.
\end{remark}

The domain $\Omega$ is a Riemann surface whose universal cover is the
unit disk, hence $\Omega$ is equipped with a Poincar\'e metric
$d_\Omega$. We will prove that all inverse branches $D^{-i}_r(z)$ of
\emph{sufficiently   protected} disks, i.e. disks $D_r(z) \subset D_{2K  \cdot r}(z) \subset \Omega$ for a sufficiently large constant $K$ to be determined later, have Poincar\'e diameter bounded by
$$
\mathrm{diam}_{\mathrm{max}} := N_0(2K) \cdot C(\frac{1}{2}, \mathrm{deg}_{\mathrm{crit}}).
$$
By choosing $\Omega$ sufficiently thin, i.e. the
constant $\epsilon$ in Lemma  \ref{1D:Omega} sufficiently small, it therefore follows that the Euclidean diameter of each $D^{-i}_r(z)$ is arbitrarily small, \emph{unless} $D^{-i}_r(z)$ is contained in a wandering Fatou component. In that case we cannot control the Euclidean diameter of $D^{-i}_r(z)$ by making $\Omega$ thinner.

If $D_r^{-i}(z)$ does have sufficiently small Euclidean diameter, and is bounded away from the critical values, then it follows that the map $f: D_r^{-i-1}(z) \rightarrow D_r^i(z)$ is univalent. Notice that simple connectivity is not used here.

By choosing $\Omega$ sufficiently thin we can guarantee that there are only finitely many wandering domains for which the bound on the Poincar\'e diameter domain does not imply the necessary bound on the Euclidean diameter.

\begin{defn}[\emph{domain with hole}]
We say that a wandering domain $U$ is a \emph{domain with hole} if there exist a domain $V \subset U$
with
$$
\mathrm{diam}_{\Omega}(V) \le N_0(2)
\cdot C(\frac{1}{2}, \mathrm{deg}_{\mathrm{crit}})
$$
for which $f: V^{-1} \rightarrow V$ is not univalent. Note that in particular any wandering domain that contains a critical value is a wandering domain with hole.
\end{defn}

Since the wandering domains are all disjoint and contained in a bounded region, it follows from Lemma \ref{1D:Omega} that if $\Omega$ is chosen sufficiently thin, then there are only finitely many wandering domains with hole.

\begin{defn}[\emph{critical wandering domain}]
Given a bi-infinite orbit of wandering components $(U^j)$, we say that a wandering domain $U^j$ is critical if $U^{j+1}$ is a wandering component with hole, but $U^i$ is not for $i \le j$.
\end{defn}

Since there are only finitely many domains with hole, there are also only finitely many critical domains. We say that a wandering domain is \emph{post-critical} if it is contained in the forward orbit of a critical domain, and \emph{regular} if it is not. Thus wandering domains in a grand orbit that does not contain critical components are all called regular.

\begin{defn}[\emph{$\mathrm{deg}_{\mathrm{max}}$}]
Since there are only finitely many critical domains, and Baker's Lemma implies that for each orbit $(U^n)_{n \in \mathbb Z}$ of wandering domains we have that $f: U^n \rightarrow U^{n+1}$ is univalent for $n$ sufficiently large, it follows that there exist an upper bound on the degree of all maps $f^n: U^0 \rightarrow U^n$ for $U^0$ critical. We denote this upper bound by $\mathrm{deg}_{\mathrm{max}}$.
\end{defn}

\subsection{Disks deeply contained in wandering domains}

Let us write $(U^n)_{n \in \mathbb Z}$ for a bi-infinite orbit of wandering Fatou components, i.e. $f(U^n) = U^{n+1}$. We will separate several distinct cases. The simplest case occurs when $D_r(z)$ is contained in what we have called a regular component.

\begin{lemma}\label{1D-lemma3}
Let $U$ be a regular wandering component, and consider a protected disk $D_r(z) \subset D_{2r} \subset U$. Then
$$
\mathrm{deg}\left(f^j: D^{-j}_r(z) \rightarrow D_r(z)\right) = 1
$$
and
$$
\mathrm{diam}_{\Omega} \left(D^{-j}_r(z)\right) \le N_0(2) \cdot C(\frac{1}{2}, 1)
$$
for all $j \ge 0$.
\end{lemma}

\begin{proof}
The proof follows by induction on $j$. Suppose that the statement holds for certain $j$, we will proceed with the proof for $j+1$. We can cover $D_r^{-1}(z)$ with at most $N_0(2)$ disks $D_{r_k}(z_k)$ satisfying
$$
D_{4 r_k}(z_k) \subset D_{2r}^{-1}(z) \subset U^{-1}.
$$
Hence we can apply the induction hypothesis for each of the disks $D_{2r_k}(z_k) \subset D_{4r_k}(z_k)$, obtaining
$$
\mathrm{deg}\left(f^j: D^{-j}_{2r_k}(z) \rightarrow D_{2r(z)}\right) = 1,
$$
and thus
$$
\mathrm{diam}_\Omega\left(D_{r_k}^{-i}(z_k)\right) \le C(\frac{1}{2}, 1)
$$
for $i = 0, \ldots ,j$.
We claim that for each $i =0 , \ldots ,j$ we obtain
$$
\mathrm{diam}_\Omega \left(D_r^{-i}(z)\right) \le N_0(2) \cdot C(\frac{1}{2}, 1),
$$
and
$$
\mathrm{deg} \left( f^i: D_r^{-i}(z) \rightarrow D_r(z)\right) = 1.
$$
The proof follows by induction on $i$. Both statements are immediate for $i = 0$. Suppose that the statements hold for $0, \ldots, i$. Since $U$ is regular, the hyperbolic diameter bound on $D_r^{-i}(z)$ implies that $f: D_r^{-i-1}(z) \rightarrow D_r^{-i}(z)$ is univalent. Thus $D_r^{-i-1}(z)$ is covered by at most $N_0(z)$ sets $D_{r_k}^{-i}(z_k)$. The diameter bound for $i+1$ follows, completing the induction step.
\end{proof}

We now consider the case when $D_r(z)$ is contained in a wandering domain $U^n$ that may be post-critical. By renumbering we may assume that $U^0$ is the critical component. The definition of $\mathrm{deg}_{\mathrm{max}}$ immediately gives diameter bounds for preimages of protected disks $D_r(z) \subset D_{K\cdot r}(z) \subset U^n$, namely
$$
\mathrm{diam}_\Omega\left(D^{-j}_r(z)\right) \le C(\frac{1}{K}, \mathrm{deg}_{\mathrm{max}})
$$
for $j \le n$. We obtain the following consequence.

\begin{lemma}\label{1D-lemma4}
We can choose $K \in \mathbb N$, independent of the wandering domain $U^n$, so that for any $D_r(z) \subset D_{K \cdot r}(z) \subset U^n$ one has
$$
\mathrm{deg}\left(f^j: D^{-j}_r(z) \rightarrow D_r(z) \right) \le \mathrm{deg}_{\mathrm{crit}}
$$
and
$$
\mathrm{diam}_\Omega\left(D^{-j}_r(z)\right) \le N_0(2) \cdot C(\frac{1}{2}, \mathrm{deg}_{\mathrm{crit}})
$$
for all $j \in \mathbb N$.
\end{lemma}
\begin{proof}
By the previous lemma, we have obtained the required estimates for $K =2$ in regular components, i.e. when $n \le 0$.

Let $n > 0$ and first consider $j \le n$. Since the degree of $f^j : U^{n-j} \rightarrow U^n$ is bounded by $\mathrm{deg}_{\mathrm{max}}$ and the disk $D_{K \cdot r}(z)$ is assumed to lie in $U^n$ it follows that
$$
\mathrm{diam}_{\Omega} D^{-j}_r(z) < C(\frac{1}{K}, \mathrm{deg}_{\mathrm{max}}).
$$
The required diameter and degree bounds follow when $K$ is chosen sufficiently large.

When $j > n$ we cannot assume a uniform degree bound on the maps $f^{j}: U^{n-j} \rightarrow U^{-n}$. However, since
the domain $U^0$ is simply connected, there is a universal bound from below on the Poincar\'e distance from the point $w \in U^0$ to the circle centered at $w$ of radius $\frac{1}{2}d(w,\partial U^0)$. Choosing $K$ such that $\mathrm{diam}_{U^0} D^{-n}_r(z)$ is strictly smaller than this universal bound implies that $D^{-n}_r(z)$ is contained in a disk $D_{\rho}(w)$ for which $D_{2\rho}(w) \subset U^0$. The statement of the previous lemma completes the proof.

\end{proof}

\medskip

\subsection{Disks that are not deeply contained.}

We restate and prove the main one-dimensional result using the constant $M = 2K$.

\medskip

\noindent{\bf Proposition \ref{1D-prop1}} \emph{
Let $z \in \Omega$ and let $r > 0$ be such that $D_{2K \cdot r}(z) \subset \Omega$. Then for every $j \in \mathbb N$ we have that
$$
\begin{aligned}
\mathrm{\bf deg}(j): & \; \; \; \; \;\mathrm{deg} \left(f^j: D^{-j}_r(z) \rightarrow D_r(z)\right) \le \mathrm{deg}_{\mathrm{crit}},\\
\mathrm{\bf diam}(j): & \; \; \; \; \; \mathrm{diam}_{\Omega} D^{-j}_r(z) \le \mathrm{diam}_{\mathrm{max}}.
\end{aligned}
$$
}
\begin{proof}
We assume the statement holds for given $j$, and proceed to prove the statement for $j+1$.

Recall that we have already proved the proposition, with a stronger diameter estimate, in the case where $D_{K\cdot r}(z)$ is contained in a wandering Fatou component. Thus, we are left with two possibilities: either $D_r(z)$ is not contained in a wandering Fatou component, or $D_r(z)$ is contained in a wandering component $U^n$ but $D_{K\cdot r}(z)$ is not. We will first prove the induction step for the former case, where $D_r(z)$ is not contained in a wandering domain. The conclusion for the former case will be used in the proof of the latter case.

Suppose $z_0 \in D_r(z) \subset D_{2Kr}(z)$ is not contained in a wandering component. Then, by making $\Omega$ sufficiently thin, it follows that any backward image of $z_0$ can be assumed to lie arbitrarily close to $\partial \Omega$.

Cover $D_r^{-1}(z)$ by at most $N_0(2K)$ disks $D_{r_k}(z_k)$ for which $D_{4Kr_k}(z_k) \subset D^{-1}_{2Kr}(z)$. The induction hypothesis gives that
$$
\mathrm{deg}\left( f^j: D^{-j}_{2r_k}(z_k) \rightarrow D_{2r_k}(z_k)\right) \le \mathrm{deg}_{\mathrm{crit}},
$$
and hence
$$
\mathrm{diam}_\Omega D^{-j}_{r_k}(z_k) \le C(\frac{1}{2}, \mathrm{deg}_{\mathrm{crit}}).
$$
It follows by induction on $i$, for $i = 0, \ldots ,j$, that
$$
\mathrm{diam}_\Omega D_r^{-j-1}(z) \le N_0(2K) \cdot C(\frac{1}{2}, \mathrm{deg}_{\mathrm{crit}}),
$$
and
$$
\mathrm{deg} \left(f^j : D_r^{-j-1}(z) \rightarrow D_r(z)\right) = 1.
$$
Here we used that each $D_r^{-i}(z)$ contains exactly one $i$-th preimage of $z_0$, denoted by $z_0^{-i}$, and by making $\Omega$ sufficiently thin the hyperbolic distance between the preimages of $z_0^{-i}$ can be assumed to be strictly larger than twice the proven diameter bound. Thus, we have completed the proof in the case where $D_r(z)$ is not contained in a wandering domain.

In the remainder of this proof we will therefore assume that $D_r(z)$ is contained in a Fatou component $U^n$, but the larger disk $D_{Kr}(z)$ is not. Let $w \in \partial U^n$ be such that $|z - w|$ is minimal, and write $[z, w] \subset \mathbb C$ for the closed interval, see Figure \ref{figure:disks}.

\begin{figure}[t]
\centering
\includegraphics[width=2.5in]{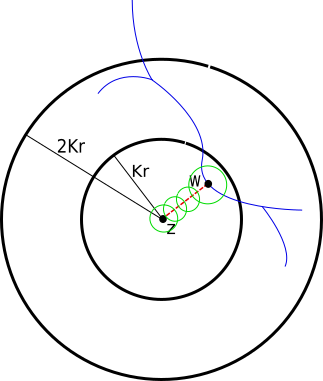}
\caption{The interval $[z,w]$ covered by $N_1$ smaller disks}
\label{figure:disks}
\end{figure}

Since $D_{Kr}(w) \subset D_{2Kr}(z)$ it follows that $D_{Kr}(w) \subset \Omega$. Hence the disk $D_{\frac{r}{2}}(w)$ satisfies the conditions of the previously discussed case, and we obtain the estimates
$$
\mathrm{diam}_{\Omega} D^{-i}_{\frac{r}{2}}(w) \le N_0(2K) \cdot C(\frac{1}{2}, \mathrm{deg}_{\mathrm{crit}}).
$$
The interval $[z, w]$ can be covered by the disk $D_{\frac{r}{2}}(w)$ and a bounded number of disks $D_{s_1}(w_1), \ldots,  D_{s_{N_1}}(w_{N_1})$ satisfying $D_{Ks_\nu}(w_\nu) \subset U^n$ for a universal constant $N_1 \in \mathbb N$. Thus the disks $D_{s_\nu}(w_\nu)$ satisfy the conditions of Lemma \ref{1D-lemma4}, and it follows that
$$
\mathrm{diam}_\Omega D^{-i}_{s_\nu}(w_\nu) \le \mathrm{diam}_{\mathrm{max}}.
$$
Thus, we obtain a bound from above on the hyperbolic distance of each $D^{-i}_r(z)$ to $\partial U^{n-i}$. By choosing $\Omega$ sufficiently thin, the bounds on the hyperbolic distance to the boundary gives arbitrarily small bounds on the Euclidean distance to the boundary, which in turn means that hyperbolic diameter estimates give arbitrarily small bounds on the Euclidean diameters of the disks $D^{-i}_{r_k}(z_k)$. We can conclude the argument by using the same induction on $i$ used in the previously discussed cases.
\end{proof}

\subsection{Consequences}

The obtained degree and diameter estimates imply a number of consequences. The first is the non-existence of wandering domains.

\begin{lemma}
There are no wandering Fatou components.
\end{lemma}
\begin{proof}
Suppose $U$ is a wandering Fatou component. We can construct the domain $\Omega$ as above sufficiently thin so that the Poincar\'e diameter of $U$ can be made arbitrarily large. In particular, we can find a relatively compact $K \subset U$ whose Poincar\'e diameter in $\Omega$ is strictly larger than $\mathrm{diam}_{\mathrm{max}}$. Let $n_j$ such that $f^{n_j}(K)$ converges to a point $p \in J$. Without loss of generality we may assume that $p$ does not lie in a parabolic cycle.  Let $D_r(p)$ be such that $D_{2r}(p) \subset \Omega$. Then $f^{n_j}(K) \subset D_r(p)$ for sufficiently large $j$, which by Proposition \ref{1D-prop1} implies that $\mathrm{diam}(K) < \mathrm{diam}_{\mathrm{max}}$, giving a contradiction.
\end{proof}

Lacking wandering domains the proof of Proposition \ref{1D-prop1} becomes considerably simpler. An immediate consequence is the following.

\begin{corollary}\label{cor:12}
The constant $\mathrm{deg}_{\mathrm{crit}}$ can be taken equal to $1$, and the constant $M$ equal to $2$.
\end{corollary}

\begin{prop}\label{prop:contraction}
Let $z \in J$ not be contained in a parabolic cycle, choose $r> 0$ so that $D_r(z) \subset D_{2r}(z) \subset \Omega$, and let $V^1, V^2, \ldots$ be such that $f(V^1) = D_r$ and $f(V^{n+1}) = V^n$. Then
$$
\mathrm{diam}_\Omega V^n \rightarrow 0.
$$
\end{prop}
\begin{proof}
By Corollary \ref{cor:12} the maps $f^n: V^n \rightarrow D_r$ are all univalent, hence we can consider the inverse branches $(f^n)^{-1}: D_r \rightarrow V^n$. These inverse branches form a bounded, and thus normal, family of holomorphic maps. Let $h$ be a limit map. Since $z \in J$ and $J$ is invariant, the image $h(D_r)$ must contain a point $q$ in $J$. But since any neighborhood of $q$ contains points in the basin of infinity, $h(D_r)$ cannot contain an open neighborhood of $q$, and hence $h$ is constant.
\end{proof}

The non-existence of rotation domains is an immediate consequence.

\begin{corollary}
The polynomial $f$ does not have any Siegel disks.
\end{corollary}

\begin{corollary}
If $f$ does not have any parabolic cycles then $f$ is hyperbolic.
\end{corollary}
\begin{proof}
When $f$ lacks parabolic cycles the set $J$ is contained in $\Omega$. Since $J$ is compact it follows from Proposition \ref{prop:contraction} that there exist $N \in \mathbb N$ and $r >0$ such for any $z \in D$ and any connected component $V^N$ of $(f^n)^{-1}(D_r(z))$ the Euclidean diameter of $V^N$ is less than $r/2$. The Schwartz Lemma therefore implies that $|(f^N)^\prime| \ge 2$ on $J$.
\end{proof}

\begin{corollary}
The polynomial $f$ lies on the boundary of the hyperbolicity locus.
\end{corollary}
\begin{proof}
By \cite{DH} the number of parabolic cycles is bounded by the degree of $f$. It follows that we can perturb the parameters $f$ slightly in a direction where $J$ changes continuously, and so that all parabolic cycles split into repelling and attracting cycles. If the perturbation is sufficiently small then there are still no critical points on $J$, while the parabolic cycles have disappeared. Hence by the previous Corollary the perturbed function is hyperbolic.
\end{proof}

We note that the bound on the number of parabolic cycles in terms of the degree is not needed if one allows perturbations into the infinite dimensional space of polynomial like maps.

\section{H\'enon maps: Background and preliminaries}

Recall from \cite{FM1989} that the dynamical behavior of a polynomial automorphism of $\mathbb C^2$ is either dynamically trivial, or the automorphism is conjugate to a finite composition of maps of the form
$$
(x,y) \mapsto (p(x) + b \cdot y, x),
$$
where $p$ is a polynomial of degree at least $2$ and $b \neq 0$. We will refer to such compositions as H\'enon maps. Given $R>0$ we define the following sets.
$$
\begin{aligned}
\Delta^2_R &:= \{(x,y)  :  |x|, |y| < R\},\\
V^+ &:= \{(x,y)  :  |x| \ge \max(R, |y|)\}, \; \; \mathrm{and}\\
V^- &:= \{(x,y)  :  |y| \ge \max(R, |x|)\}
\end{aligned}
$$
By choosing $R$ sufficiently large we can make sure that $f(V^+) \subset V^+$, $f^{-1}(V^-) \subset V^-$ and $f(\Delta^2_R) \subset \Delta^2_R \cup V^+$. One can also guarantee that if $(x_0,y_0) \in V^+$ and $(x_1, y_1) = f(x_0,y_0)$ then $|x_1| > 2|x_2|$. Similarly one obtains $|y_{-1}|> 2|y_0|$ for $(x_0,y_0) \in V^-$. It follows that every orbit that lands in $V^+$ must escape to infinity, and every orbit that does not converge to infinity must eventually land in $\Delta^2_R$ in a finite number of steps.

\medskip

We write $K^+$ for the set with bounded forward orbits, $K^-$ for the
set with bounded backward orbits, and $K=K^+ \cap K^-$. As usual we
define the \emph{forward and backward Julia sets} as
$J^\pm= \partial K^\pm$,
and the Julia set as $J = J^+ \cap J^-$. Let us recall the existence of the Green's currents $T^+$ and $T^-$, supported on $J^+$ and $J^-$, and the equilibrium measure $\mu = T^+ \wedge T^-$, whose support $J^\star$ is contained in $J$. Whether $J^*$ always equals $J$ is one of the main open questions in the area, and was previously only known for hyperbolic H\'enon maps \cite{BS1991a}.

\subsection{Wiman Theorem and Substantially dissipative H\'enon maps}

Recall that a subharmonic function $g: \mathbb C \rightarrow \mathbb
R$ is said to have order of growth at most $\rho$
if
$$
g(\zeta) = O(|z|^\rho)\quad  {\mathrm{as}} \quad \zeta \rightarrow \infty.
$$

Given $ E  \in \R$, let us call the set $\{ g < E \} $ {\em subpotential}
(of level $E$),
and its components {\em subpotential components}.

\begin{thm}[Wiman]
Let $g$ be a non-constant subharmonic function with order of growth
strictly less than $\frac{1}{2}$.
Then subpotential components of any level $E$ are bounded.
\end{thm}

Let us describe how it was
used in the setting of H\'enon maps in recent works of the first author and Dujardin \cite{DL2015}, and in \cite{LP2014}. Suppose that $p \in \mathbb C^2$ is a hyperbolic fixed point, and let $W^s(p)$ be its stable manifold, corresponding to the stable eigenvalue $\lambda$. Then there exists a linearization map $\varphi: \mathbb C \rightarrow W^s(p)$ satisfying $\varphi (\lambda \cdot \zeta) = f(\varphi(\zeta))$. As usual we let $G^\pm$ be the plurisubharmonic functions defined by
$$
G^\pm(z) = \lim_{n \rightarrow \infty} \frac{1}{ ( \deg f )^n} \cdot \log^+\|f^{\pm n}(z)\|.
$$
We have the functional equations
$$
G^\pm(f(z)) = (\mathrm{deg} f)^{\pm 1} \cdot G^\pm(z).
$$
Combining the functional equations for $G^-$ and $\phi$ we obtain that the non-constant subharmonic function $g = G^- \circ \varphi$ satisfies
$$
g(\lambda \cdot \zeta)= G^- \circ f \circ \varphi(\zeta) =
\frac{1}{\deg f} g(\zeta).
$$
Note that $|\lambda| < |\mathrm{Jac} f|$. Hence under the assumption
that $|\mathrm{Jac} f| < \frac{1}{\mathrm{deg} f^2}$ it follows that
$g$ is a subharmonic function of growth strictly less than
$\frac{1}{2}$, and therefore according to Wiman's Theorem
all its subpotential  components 
are bounded.

\medskip

Let us point out that the above discussion also holds when $p$ is a neutral fixed point, i.e. having one neutral and one attracting multiplier. One considers the strong stable manifold with corresponding eigen value $\lambda$, satisfying $|\lambda| = |\mathrm{Jac} f|$. The subharmonic function $g$ still has order of growth strictly less than $\frac{1}{2}$. The idea can also be applied when $p$ is not a periodic point but lies in a invariant hyperbolic set, or under the assumption of a dominated splitting, which will be discussed in the next section.

\medskip

A particular consequence of Wiman's Theorem is that all connected components of intersections of (strong) stable manifolds with $\Delta^2_R$ are bounded in the linearization coordinates. By the Maximum Principle they are also simply connected, hence they are disks. The filtration property of H\'enon maps tells us that every connected component of the intersection of a stable manifold with $\Delta^2_R$ is actually a branched cover over the vertical disk $\Delta_{w}(R)$. This will be used heavily in what follows.

\subsection{Classification of periodic components}

\subsubsection{Ordinary components}

We recall the classification of periodic Fatou components $U$ from
\cite{LP2014}, building upon results of Bedford and Smillie
\cite{BS1991b}.
For a dissipative H\'enon map $f$,
there exist three types of {\em ordinary} invariant%
\footnote{A description of periodic components readily follows.
Note also that  since  $f$ is invertible,
there is no such thing as a ``preperiodic'' Fatou component.}
Fatou components $U$:

\begin{enumerate}
\item[(i)]  {\em Attracting basin:} All orbits in $U$ converge to an attracting  fixed point $p\in U$.
Moreover, $U$ is a Fatou-Bieberbach domain (i.e., it is  biholomorphically equivalent to $\mathbb C^2$).

\item[(ii)] {\em Rotation basin:}
All orbits in $U$ converge to a properly embedded Riemann  surface $\Sigma \subset U$,
which is invariant under $f$ and biholomorphically equivalent to either an
annulus or the unit disk. The biholomorphism from $\Sigma$ to an annulus or disk can be chosen so that it conjugates the action of $f|_\Sigma$ to an irrational rotation. The stable manifolds through points in $\Sigma$ are all embedded complex lines, and the domain $U$ is biholomorphically equivalent to $\Sigma \times \mathbb C$.

\item[(iii)]  {\em Parabolic basin:}
All orbits in $U$ converge to a parabolic fixed point $p  \in \partial U$ with
the neutral eigenvalue equal to $1$. Moreover, $U$ is a Fatou-Bieberbach domain.
\end{enumerate}

In the literature a periodic point whose multipliers $\lambda_1$ and $\lambda_2$ satisfying $|\lambda_1|<1$ and $\lambda_2 = 1$ may be called either semi-parabolic or semi-attracting, depending on context. Since we are working with dissipative H\'enon maps, where there is always at least one attracting multiplier, we chose to refer to these points as \emph{parabolic}, and we use analogues terminology for Fatou components. Similarly, we will call a periodic point with one neutral multiplier \emph{neutral}.

In each  case, we let $A= A_U $ be the {\em attractor} of the
corresponding component (i.e., the attracting or parabolic
point $p$, or the rotational curve $\Si$).

Along with global Fatou components $U$, we will consider \emph{semi-local}
ones,  which are components $U^i$  of the intersection $U \cap \De^2_R$.
(Usually there are infinitely many of them.)  Each $U^i$ is mapped
under $f$ into some component $U^j$, $j= j(i)$, with ``vertical'' boundary $\partial U^i \cap \Delta^2_R$ being mapped into the vertical boundary of $U^j$,
 but the correspondence $i\mapsto j(i)$ is not in general injective.
This dynamical tree of semi-local components resembles closely the
one-dimensional picture.
In particular,  cycles of ordinary  semi-local components can be viewed as
the \emph{immediate basins}  of the corresponding attractors $A_U$.

\msk
\subsubsection{Absorbing domains}

Given a compact subset $Q \subset U$,
let us say that an invariant domain $D\subset U$ is $Q$-{\em absorbing}
if there exist a moment $n\in \N$ such that  $f^n (Q) \subset D$.
If this happens for any $Q$ (with $n$ depending on $Q$) then $D$ is
called {\em absorbing}.

For instance, in the attracting case, any forward invariant neighborhood of the attracting point is absorbing. In the parabolic case,
there exists an arbitrary small  absorbing ``attracting petal''
$P\subset U$ with $\di P \cap \di U = \{ p\}$  (see
\cite{BSU2012}).  To construct a $Q$-absorbing domain in the rotation
case,   take a sufficiently large invariant
subdomain $\Si'\subset \Si$ compactly contained in $\Si$   and let
$$
       D= \bigcup_{z\in \Si'}  W_{\loc}^s (z).
$$

This  implies:

\begin{lemma}\label{absorbing sets}
Let $f$ be a dissipative H\'enon map, and let
$U$ be an ordinary invariant  Fatou component with an attractor $A$.
Then any compact set $Q\subset U$ is contained in a forward invariant
 domain  $W\subset U$ such that $\overline W \cap \di U \subset A$.
(In particular,  $W$ is relatively compactly contained in $U$
in the attracting or rotation cases.)
\end{lemma}

Let us say that a subset $\Om\subset
\De^2_R$ is {\em relatively backward invariant} if
$$
f^{-1}(\Omega) \cap \Delta^2_R \subset \Omega.
$$

\begin{corollary}\label{cor:absorbing}
Given any compact set $Q$ contained in the union of periodic Fatou
components, there exists an open and relatively backward invariant
subset $\Om$ of $\De^2_R$ containing
$$
   \De^2_R \cap ( \mathrm{wandering\ components} \cup J^+ \sm  \overline{ \{ {\mathrm{ parabolic\ cycles}} \}} )
$$
and avoiding $Q$.
\end{corollary}
\begin{proof}\label{neibhd W}
There exist only finitely many periodic Fatou components $U^i$ intersecting $Q$.
For each of them, let $W_i$ be the neighborhood of $(Q\cap U^i)$ from Lemma~\ref{absorbing sets}.
Note that $J^+\cap (\bigcup \overline{W_i})$ is
contained in a finite number of parabolic cycles.
Take now a small $\epsi>0$ and let
$$
       \Om = \De^2_R \cap \{ G^+ < \epsi \} \sm  (\bigcup \overline{W_i}) .
$$
\end{proof}

\msk
\subsubsection{Substantially dissipative maps}

\begin{thm}[\cite{LP2014}]
 For a substantially dissipative H\'enon map,
any periodic Fatou component is ordinary.
\end{thm}

\begin{remark}
In fact, for H\'enon maps with dominated splitting,
this classification holds without assuming that dissipation is substantial,
see Proposition \ref{per comps} below.
\end{remark}

\section{Dominated splitting}

\subsubsection{Definition}
We say that a H\'enon map $f$ admits a dominated splitting if there is an invariant splitting of the tangent bundle on $J = J^+ \cap J^-$
\begin{equation}\label{initial line fields}
T_J(\mathbb C^2) = E^s \oplus E^c
\end{equation}
with constants $0<\rho<1$ and $C>0$ such that for every $p \in J$ and any unit vectors $v \in E^s_p$ and $w \in E^c_p$ one has
$$
\frac{\|df^n v\|}{\|df^n w\|} < C \cdot \rho^n.
$$

{\em From now on we will assume that $f$ is dissipative, from which it immediately follows that $E^s$
is stable.}
We cannot conclude that $E^c$ is unstable, though.

\msk

\subsubsection{Cone fields} Let $p \in J$ and let $v \subset E^s_p$ have unit length. Given $0<\alpha<1$ we can define the cone
$$
C_p^s(\alpha) := \{ w \in T_p(\mathbb C^2) : |<w,v>| \ge \alpha \|w\|\}.
$$
It follows from the dominated splitting that we can choose $\alpha$ continuously, depending on $p \in J$, so that
$$
df C_p^s(\alpha(p)) \supset C_{f(p)}^s(r \cdot \alpha(f(p)))
$$
for some $r<1$ which can be chosen independently of $p \in J$. We
refer to the collection of these cones as the (backward) invariant
\emph{vertical cone field} on $J$.

\medskip

Since both $E^s_p$ and $\alpha$ vary continuously with $p$, and
the set $J$ is closed, we can extend the vertical cone field
continuously to $\Delta^2_R$. It follows automatically that the
extension of the cone field to $\Delta^2_R$ is backward invariant for
points lying in a sufficiently small neighborhood $\mathcal{N}(J)$.

Note that all accumulation points of the forward orbit of a point in $J^+$ must lie in $K^- = J^-$, and therefore in $J = J^+ \cap J^-$. Writing $J^+_R = J^+ \cap \Delta^2_R$ as before, it follows from compactness that there exists an $N \in \mathbb N$ such that $f^n(J^+_R) \subset \mathcal{N}(J)$ for all $n \ge N$. Thus we can pull back the vertical cone field to obtain a backwards invariant cone field on a neighborhood of $J^+_R$. We will denote this neighborhood by $\mathcal{N}(J^+_R)$, and refer to it as the region of dominated splitting.

\subsubsection{Strong stable manifolds} \label{stable manifolds}
Let us consider the following completely invariant set:
\begin{equation}\label{VV}
   \VV^+ : = \{ p :  \  \exists \ n_0= n_0 (p) \ f^n p \in \NN (J) \
   {\mathrm for} \ n\geq n_0 \}.
\end{equation}
Let $U(p)$ be a small ball centered at $p$.  For $p\in \VV^+$,
consider a straight complex line through $f^n(p)$ whose tangent space
at $f^n(p)$ is contained in the vertical cone,
and pull back this line by $f^n$, keeping only the connected component
through $p$ in the neighborhood $U(p)$.
By the standard graph transform method, this sequence of holomorphic
disks converges to a complex submanifold $W^s_\loc (p)$, the so-called
{\em local strong stable manifold}  through $p$. By pulling back the
local stable manifolds through $f^n(p)$ by $f^{-n}$
we obtain in the limit the {\em global strong stable manifold}
through $p$, denoted by $W^s(p)$.
In line with our earlier introduced notation we will write $W^s_R(p)$ for the connected component of $W^s(p)\cap \Delta^2_R$ that contains $p$, and refer to it as a \emph{semi-local} stable manifold.

We will refer to the collection of these semi-local strong stable
manifolds as the \emph{(semi-local) dynamical vertical lamination}.

For $p\in \VV^+$, we let $E^s_p$ be the tangent line to $W^s (p)$, which can be also constructed directly as
$$
E^s_p = \bigcap_{n\geq 0} df^{-n} C^s_{f^n p} (\alpha).
$$
The lines $E^s_p$ form the {\em stable line field} over $\VV^+$,
extending the initial stable line (\ref{initial line fields})  field over $J$.

Similarly to $\VV^+$ we can consider
\begin{equation}
\VV^- : = \{ p :  \  \exists \ n_0= n_0 (p) \ f^{-n} p \in \NN (J) \ {\mathrm for} \ n\geq n_0 \}.
\end{equation}
For $p \in \VV^-$ we cannot guarantee the existence of a horizontal center manifold, but there does exist a unique \emph{central line field}, i.e. a tangent subspace whose pullback under $f^n$ is contained in the horizontal cone field for all $n\ge n_0$. For points $p \in \VV^+ \cap \VV^-$ we can consider both the vertical and the central line field. Tangencies between those two line fields play the role of critical orbits. By the dominated splitting these can only occur for orbits that leave and come back to the domain of dominated splitting. A major part of this paper is aimed at obtaining a better understanding of such tangencies.

\subsubsection{Linearization coordinates}

The global strong stable manifolds $W^s(p)$ of points $p$ in the dynamical vertical lamination can be uniformized as follows. Denote by $\pi_p: W^s(p) \rightarrow T_p(W^s(p))$ the projection to the tangent plane. The projection is locally a biholomorphism, as local stable manifolds are graphs over the tangent plane. The size of the local stable manifolds can be taken uniform over all $p$ in the dynamical vertical lamination.
Define $\varphi_p: W^s(p) \rightarrow T_p(W^s(p))$ by
$$
\varphi_p = \lim_{n \rightarrow \infty} [Df^n (p)]^{-1} \circ \pi_{f^n(p)} \circ f^n.
$$
Identifying the tangent plane with $\mathbb C$ we can view $\varphi_p$ as a biholomorphic map from $W^s(p)$ to $\mathbb C$. This identification is canonical up to a choice of argument. The identifications can locally be chosen to vary continuously with $p$. As the tangent planes to the dynamical vertical lamination vary continuously with $p$, and the above convergence to $\varphi_p$ is uniform over $p$ in the dynamical vertical lamination, one can locally obtain a continuous family of linearization maps $\varphi: W^s(p) \rightarrow \mathbb C$.

The composition of the Green's function $G^-$ with the
 linearization map gives a subharmonic function on the $\mathbb
 C$-coordinates of $W^s(p)$, which, provided the neighborhood
 $\mathcal{N}_R(J^+)$ is made sufficiently thin, has order of growth strictly less than $\frac{1}{2}$. Hence for
each point $p \in \mathcal{N}_R(J^+)$ the local stable manifold
$W^s_R(p)$
is a properly embedded disk in $\Delta^2_R$, with the projection to the second coordinate giving branched covers of uniformly bounded degrees.

\subsubsection{Fatou components}

While the substantial dissipativity assumption plays an
important role in the current paper,
the bound on the Jacobian in terms of the degree is not needed for the classification of periodic Fatou components
in the dominated splitting setting:

\begin{prop} \label{per comps}
For a dissipative  H\'enon map with dominated splitting,
any periodic Fatou component is an ordinary component.
\end{prop}

\begin{proof}
In \cite{LP2014} the assumption that the H\'enon map is substantially
dissipative plays a role in only an isolated part of the proof, namely
to prove the uniqueness of limit sets on non-recurrent
Fatou components. We note that in order to prove this uniqueness, one
does not need to assume substantial dissipativity for H\'enon maps
admitting a dominated splitting. Recall that the only point in the
proof where substantial dissipativity is used, is to rule out a
one-dimensional limit set $\Sigma$
contained in the strong stable
 manifold of a hyperbolic or neutral fixed point. Suppose that there exists a dominated splitting near $J$, and that such a $\Sigma$ does exist. As was pointed out in \cite{LP2014}, the restriction of $\{f^n\}$ to $\Sigma$ is a normal family. Recall also that $\Sigma$ must lie in $J$, hence
through each point    $q \in \Sigma$ there exists a strong stable manifold $W^s(q)$.
If  $\Si$ is transverse to   the stable field $\{E^s\}$  at some point
$q\in \Si$,
 then the union of the stable manifolds contains an open neighborhood
of $q$, on which the family of iterates is necessarily a normal family. This gives a
contradiction with $\Sigma \subset J$.
On the other hand, if $\Sigma$ is everywhere tangent to the stable field,
then for any $q\in \Si$, it is a domain in the stable manifold $W^s(q)$.
Being backward invariant, $\Si$ must coincide with $W^s(q)$.
However, $W^s(q)$ is conformally equivalent to $\C$, while $\Si$ cannot, giving a contradiction.
\end{proof}

It is a priori not clear that there are only finitely many periodic components. In the substantially dissipative case finiteness is a consequence of our main result.

\msk

\subsubsection{Rates}

We will show now that the rate of contraction on the central line bundle
 is subexponential.

\begin{lemma}\label{newlemma4}
Given any $r_1<1$ there exists a $C>0$ such that for any $p \in J$ and any unit vector $w \in E^c_p$ we have
$$
\|df^n w\| > \frac{1}{C} \cdot {r_1}^n.
$$
\end{lemma}

\begin{proof}
Let us for the purpose of a contradiction suppose that for some $r_1 < 1$ there exist for arbitrarily large $n \in \mathbb N$ unit vectors $w_n \in E^c$ with
$$
\|df^n w_n\| < r_1^n.
$$
Let $r_1 < r_2 < 1$. Then there exists an $\epsilon>0$ and for every $n \in \mathbb N$ an integer $k \in \{0, 1,\ldots n-1\}$ such that
$$
\|df^j (df^k w_n)\| < r_2^j \cdot \| df^k w_n\|
$$
for $j \le \epsilon \cdot n$. It follows that for every $m \in \mathbb N$ there exists a unit vector $u_m \in E^c$ for which
$$
\|df^j u_m\| < r_2^j
$$
for $j = 0, \ldots m$. Here $u_m$ can be chosen a multiple of a vector $df^k w_n$.

\medskip

Since the set of unit vectors in $E^c$ is compact, there exists an
accumulation point $w \in  E^c$ of the sequence $(u_m)$. Let $p \in J$ be such that $w \in E^c_p$. By continuity of the differential $df$ it follows that
$$
\|df^j w\| \le r_2^j
$$
for all $j \in \mathbb N$. Since $T_p(\mathbb C^2) = E_p^s\oplus E_p^c$, and by the definition of the dominated splitting, there exists a $C>0$ such that
$$
\|Df^j(z_0) \| < C \cdot r_2^j
$$
for all $j \in \mathbb N$. Here we have used that there is a uniform bound from below on the angle between the vertical and horizontal tangent spaces.

\medskip

Let $\xi>1$ be sufficiently small such that $\xi \cdot r_2 < 1$. By compactness of $\Delta^2_R$ there exists a $\rho>0$ such that if $x, y \in \Delta^2_R$ with $\|x - y\| < \rho$ then
$$
\|f(x) - f(y)\| \le \xi \cdot \|Df(x)\| \cdot \|x - y\|.
$$
Let $z \in \Delta^2_R$ be such that $\|z - p\|< \frac{\rho}{C}$. Then it follows by induction on $n$ that
$$
\|f^n(z) - f^n(p)\| \le \rho (\xi \cdot r_2)^n
$$
for every $n \in \mathbb N$. Hence there is a neighborhood $U$ of $p$ such that
$$
\|f^n(z) - f^n(p)\| \rightarrow 0,
$$
uniformly over all $z \in U$. But then $\{f^{n}\}_{n \in \mathbb N}$ is a normal family on $U$, which contradicts the fact that $p \in J^+$.
\end{proof}

It follows that the exponential  rate of contraction on the stable
subbundle is at least $\delta$.

\begin{lemma}\label{newlemma5}
Given any $r > |\delta|$ we can find $C>0$ such that for any unit vector $v \in E^s$ we have
$$
\|df^n v\| < C \cdot r^n.
$$
\end{lemma}
\begin{proof}
Write $v \in E^s_p$, and let $w \in E^c_p$ be a unit vector. The inequality follows immediately from Lemma \ref{newlemma4} and the fact that
$$
\|df^n v\| \cdot \|df^n w\| \le C_1 |\delta|^n,
$$
where the constant $C_1$ depends on minimal angle between the spaces $E^s$ and $E^c$.
\end{proof}

\section{Dynamical lamination and its extensions}

\subsection{Dynamical lamination}

Now let us assume that the map $f$ is substantially dissipative. Then
the rate
 $r$ in Lemma  \ref{newlemma5} can be assumed to
be strictly smaller than    $ 1/d^2 $.

It follows that for any point $p  \in J$,
 the composition $G^- \circ \varphi_p$ is a subharmonic function of
order  bounded by  $\rho < \frac{1}{2}$,
so the  Wiman Theorem can be applied.
It implies that $W^s_R(p)$ is an embedded holomorphic disk, and that the
projection to the second coordinate $\pi_2: W^s_R(p) \rightarrow \Delta_R$ gives a branched covering of finite
degree.

\begin{lemma}\label{boundeddegrees} The degrees of the branched coverings
$\pi_2: W^s_R(p) \rightarrow \Delta_R$
are uniformly bounded, and
$$
          J^+_R = \bigcup_{p \in J} W^s_R(p).
$$
\end{lemma}
\begin{proof}
Let $q \in J^+_R$.
Note that the sets
$$
V_n (q)  = \{(x,y) \in W^s_R(q) : f^{-n}(x,y) \in \Delta^2_R\}
$$
form a nested sequence of non-empty compact sets,
so they have a non-empty intersection.
Hence each $W^s_R(q)$ intersects $K^- = J^-$. Therefore we have
$$
      J^+_R = \bigcup_{p \in J} W^s_R(p).
$$

Note that the degree of $W^s_R(p)$ depends lower semi-continuously on $p$; the degree may drop at semi-local stable manifolds tangent to the boundary of $\Delta^2_R$. However, when we consider the restriction of such a stable manifold to a strictly larger bidisk $\Delta^2_{R^\prime}$, its degree, which is still finite, is at least as large as the degree of sufficiently nearby stable manifolds restricted to the smaller bidisk $\Delta^2_R$.

To argue that the degrees of the branched coverings are uniformly bounded, suppose for the purpose of a contradiction that there is a sequence $(W^s_R(p_j))$ for which the degrees converge to infinity. Without loss of generality we may assume that the sequence $(p_j)$ converges to a point $p\in J^+_R$. Let $R^\prime>R$. Then for $j$ sufficiently large the degree of $W^s_R(p_j)$ is bounded by the degree of $W^s_{R^\prime}(p)$, which gives a contradiction.
\end{proof}

We will refer to the lamination on $J^+ \cap \Delta^2_R$ given by
these local strong stable manifolds as the \emph{(semi-local) dynamical  lamination}.
In what follows we will
extend this lamination, in a non-dynamical way,   to a larger subset of $\Delta^2_R$.

\subsection{Local and global extensions of the vertical lamination}
\label{loc and glob extensions}

We note that the dynamical vertical lamination discussed previously consists of local leaves $L_R(a)$ that are connected components of global leaves $L(a)$ intersected with $\Delta^2_R$. These leaves all have natural linearization parametrizations that vary continuously with the base point $a$.

Let us recall the $\lambda$-lemma, in this version due to Slodkowski \cite{Slodkowski}.

\begin{lemma}
Let $A \subset \hat{\mathbb C}$. Any holomorphic motion $f: \mathbb D
\times A \rightarrow \hat{\mathbb C}$ of $A$ over $\mathbb D$
extends to a holomorphic motion $\mathbb D \times \hat{\mathbb C} \mapsto \hat{\mathbb C}$ of $\hat{\mathbb C}$ over $\mathbb D$.
\end{lemma}

Let $L_R(a)$ be a dynamical leaf, with a linearization map $\varphi_a: \mathbb C \rightarrow L(a)$. Recall that by the assumption that $f$ is substantially dissipative we have that $L_R(a) = \varphi_a(D)$, for some bounded simply connected set $D \subset \mathbb C$. Let $D$ be compactly contained in a slightly larger simply connected set $D^\prime$, and let $\xi: \mathbb D \rightarrow D^\prime$ be the Riemann mapping. Define $\psi_a = \varphi_a \circ \xi : \mathbb D \rightarrow L_R(a)$. Then there exists a biholomorphic map $\Psi$ from $\Delta_\epsilon \times \mathbb D \rightarrow \mathbb C^2$ with $\Psi(\zeta, 0) = \psi_a(\zeta)$, mapping to a tubular neighborhood of the ``core'' $L_R(a)$. Consider all dynamical leaves that intersect a small neighborhood $\Psi(\Delta_\delta \times \mathbb D \rightarrow \mathbb C^2)$, where $\delta$ is chosen sufficiently small so that these dynamical leaves are completely contained in $\Psi(\Delta_\epsilon \times \mathbb D \rightarrow \mathbb C^2)$. If $\epsilon$ is sufficiently small then the inverse images under $\Psi$ of these leaves form a collection of pairwise disjoint ``dynamical'' graphs over $\mathbb D$ in $\Delta_\epsilon \times \mathbb D$, thus giving a holomorphic motion of a set $A \subset \Delta_\epsilon$.

By the $\lambda$-lemma the motion extends to a holomorphic motion over $\Delta_\epsilon$. The graphs over $\mathbb D$ that are completely contained in $\Delta_\epsilon \times \mathbb D$ can be mapped back by $\Psi$. By restricting to a slightly smaller vertical disk $D_{1-\eta}$, we can guarantee that all graphs that intersect a sufficiently small neighborhood of the core $\{0\} \times \mathbb D_{1-\eta}$ are completely contained in $\Delta_\delta \times D_{1-\eta}$, and can therefore be mapped back to $\mathbb C^2$ by $\Psi$. We obtain a collection of pairwise disjoint ''graphs'' over $L_R(a)$, filling a neighborhood and all remaining in the neighborhood sufficiently close to $L_R(a)$. Moreover, by construction the newly constructed graphs cannot intersect any dynamical graphs.

We will refer to such an extension as a \emph{flow box}, and to  the leaves as {\em vertical}.
Note that the dynamical leaves were globally defined, while the new leaves in the flowboxes are only defined in $\Delta^2_R$.

By compactness the Euclidean radii of the
tubular neighborhoods of the dynamical leaves can be chosen uniformly, and hence the
dynamical vertical lamination is contained in a finite number of flow boxes. One could apply the $\lambda$-lemma to each of these, but a priori there is no reason why new leaves coming from different flow boxes should not intersect transversely. The main result in this section is Proposition \ref{prop:single}, where a single extension to a neighborhood of the dynamical vertical lamination is constructed. Let us give an outline of the argument before going into details.

\begin{figure}[t!]
\centering
\begin{subfigure}[t]{0.5\textwidth}
\centering
\includegraphics[width=0.9\linewidth]{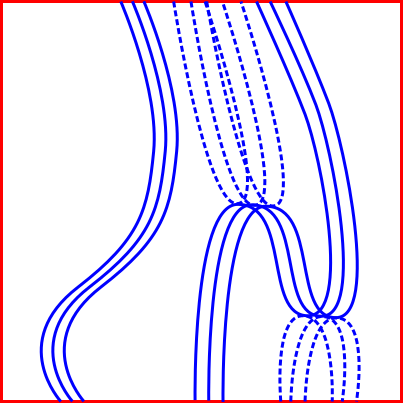}
\end{subfigure}%
~
\begin{subfigure}[t]{0.5\textwidth}
\centering
\includegraphics[width=0.9\linewidth]{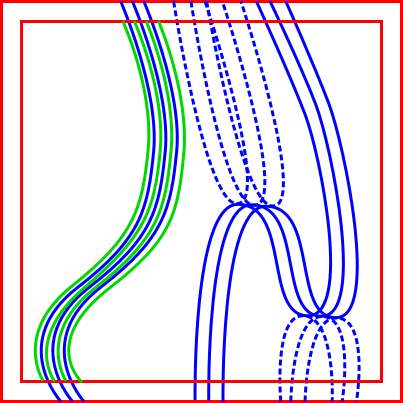}
\end{subfigure}
\caption{Local extension of the lamination on slightly smaller bidisk}
\label{figure:lamination}
\end{figure}

The extension of the lamination will be constructed by applying the
$\lambda$-lemma to a finite number of tubular neighborhoods,
each time taking into account the leaves that have been considered in
previous steps.

A difficulty is that the leaves that we construct in a local extension are not global, they only are defined in some the tubular neighborhood. In particular, even if we can guarantee that all new leaves, are graphs over the core of other tubular neighborhoods they intersect, they may not be graphs over the entire core, see Figure \ref{figure:extension2}. We can deal with this by starting with a strictly larger bidisk $\Delta^2_{R^\prime}$, and reducing the radius $R^\prime$ after each local extension by twice the radius of the tubular neighborhood. The goal is therefore to reduce the radius by at most the difference $R^\prime - R$ we start with.

Starting with an even larger constant $R^\prime$ is not of help, as that would affect the size and geometry of the flowboxes. Just reducing the radii of the flow boxes seems useless as well, as that would increase the number of flow boxes needed. The solution is to carry out the $\lambda$-lemma on large numbers of pairwise disjoint flow boxes simultaneously. By a covering lemma a la Besicovitch it follows that we can finish the process in a number of steps $N$ that is independent of the radii of the flow boxes. By starting with a slightly larger bidisk $\Delta^2_{R^\prime}$ and choosing the radii $\epsilon$ so that $N \cdot 2\times \epsilon < R^\prime - R$ we will obtain the desired extension.

The version of Besicovitch Covering Theorem that we will use relies on the fact that the dynamical vertical lamination is Lipschitz, which follows from the following:

\begin{lemma}\label{lemma:prep1}
The holonomy maps of the dynamical vertical lamination are $C^{1+\epsilon}$ smooth.
\end{lemma}
\begin{proof}
The analogous statement in the uniformly hyperbolic case was proved in
\cite{L1999}, with a similar proof. It is sufficient to prove that
under the holonomy maps induced by the dynamical vertical lamination,
the dilatation of images of small disks of radius $r>0$ is $1+ r^{\epsilon} $, for some $\epsilon>0$.

We consider holonomy between horizontal transversals through two
points $z, w$ with $w \in W^s_{loc}(z)$.
Let $\eta>0$ be such that any horizontal disk through any point $f^n(z)$ or $f^n(w)$ of radius at most $\eta$ is a graph over the horizontal tangent line.

Let $D(z)$ and $D(w)$ be horizontal transversal disks, let $0<r<\eta$, and let $\Delta_r(z) \subset D(z)$ be a graph over the disk of radius $r$. We choose $n\in \mathbb N$ be so that
$$
\|df^n|_{E^s_z}\| \sim \frac{r}{\eta}.
$$

Note that
$$
\|df^n|_{E^c_z}\| \cdot \|df^n|_{E^s_z}\| \sim |\mathrm{Jac}(f)|^n,
$$
and thus decreases exponentially fast. It follows that for $r$ sufficiently small,
the disks $\Delta^n_r(z) = f^n \Delta_r(z)$ have size $< \eta$, so
they are graphs over the horizontal tangent line at $f^n(z)$.
Hence the composition of $f^n : \Delta_r(z) \rightarrow \Delta^n_r(z)$ with the respective projections to and from the respective horizontal tangent lines at $z$ and $f^n(z)$ produces a univalent function.

Consider the disks $\Delta_{r^2}(z) \subset \Delta_r(z)$. By the Koebe Distortion Theorem, the dilatation of $\Delta^n_{r^2}(z)$ is bounded by $1+O(r)$. By our choice of $n$ it follows that $\|z^n - w^n\|$ is of order $r$. Hence the holonomy from $\Delta^n_r(z)$ to its image $\Delta^n_r(w)$ is quasiconformal of order $r$, and the dilatation of $\Delta^n_{r^2}(w)$ is bounded by $1+O(r)$.

The modulus $-\log(r)$ of the annulus $\Delta_r(z) \sm \Delta_{r^2}(z)$ is preserved under conformal maps, hence $\mathrm{Mod}(\Delta^n_r(z) \sm \Delta^n_{r^2}(z)) = -\log(r)$. Since the distance between the disks $\Delta_r(z)$ and $\Delta_r(w)$ is of order $r$, the holonomy map from one to the other changes the modulus of the respective annuli by at most a factor of order $1-r$, hence
$$
\mathrm{Mod}(\Delta^n_r(w) \sm \Delta^n_{r^2}(w)) = \geq ((1-O(r) )
\cdot \log(\frac{1}{r}) ) = \geq \log(\frac{1}{r^{1- O(r) }}).
$$
Therefore,
we can again apply the Koebe Distortion Theorem to conformal map
$f^{-n} : \Delta^n_r(w) \rightarrow \Delta_r(w)$,
and it follows that the dilatation of $\Delta^n_{r^2}(w)$ is bounded
by $1 + r^{1-O(r) } $.

Note that restricted to the dynamical vertical lamination, the
holonomy maps that we considered by mapping back and forth by $f^n$
are all equal, hence the dilatation bound of $1+ r^{1-O(r) } $ applies to
the holonomy in the original flow box as well. Letting $r^2 = \rho$,
it follows that the dilatation on a disk of radius $\rho$ is
$1+\rho^\epsilon $, for $\epsilon>0$ that can in fact be chosen arbitrarily close to $\frac{1}{2}$. This completes the proof.
\end{proof}

In the real differentiable setting there have been a number of results regarding the smoothness of holonomy maps in the partially hyperbolic setting, see for example (\cite{PSW1997,PSW2000}). Often smoothness holds when the center eigen values are sufficiently close to each other, i.e. satisfy some ``center-bunching condition''. Such condition is trivially satisfied when the center direction is one-dimensional, or as here, in the conformal setting.

\medskip

We note that the above proof shows that the holonomy $C^{1+\epsilon}$ on the dynamical vertical lamination for any $\epsilon< \frac{1}{2}$. We will not use this estimate. In fact, we will only use that the dynamical vertical lamination is Lipschitz.

Note that in later steps of the procedure, after having already
found a partial extension of the dynamical vertical lamination by a number of
applications of the $\lambda$-lemma, the lamination under
consideration may no longer be Lipschitz. However, we will see that this does not
present difficulties when only the leaves in sufficiently small neighborhoods of dynamical leaves are kept.

\begin{defn} Since the holonomy maps of the dynamical vertical lamination are Lipschitz, there exists a constant $k>0$, independent of $\epsilon>0$ for $\epsilon$ sufficiently small, such that any dynamical leaf that intersects an $\epsilon$-tubular neighborhood of a dynamical leaf must be contained in the corresponding $(k\cdot \epsilon$)-tubular neighborhood. Note that $\epsilon$ and $k\cdot \epsilon$ refer to the \emph{Euclidean} radius of the tubular neighborhoods in $\mathbb C^2$.
\end{defn}

For given $\epsilon>0$ we will consider tubular neighborhoods of three different radii: $\epsilon$, $k \cdot \epsilon$ and $k^2\epsilon$. Given a collection of leaves $\{L_R(a_i)\}$, we will denote the tubular neighborhood of radius $r_i$ centered at $L_R(a_i)$ as $T_i(r_i)$. In what follows we consider tubular neighborhoods in different bidisks $\Delta^2_{R^\prime}$, where $R^\prime > R$ decreases in each step. Without loss of generality we may assume that the constant $k>0$ defined above will be sufficiently large for the maximal bidisk $\Delta^2_{R^\prime}$ as well. We will write $T_i(r_i, R^\prime)$ to clarify the radius $R^\prime$ of the bidisk we consider.

Let us fix a straight horizontal line
$$
\mathbb L_0 = \{(x,y) : y =y_0, \}
$$
for some $|y_0|< R$.

\begin{lemma}\label{lemma:prep2}
The dynamical vertical lamination has only finitely many horizontal tangencies in $\mathbb L_0$.
\end{lemma}
\begin{proof}
Apply the $\lambda$-lemma to a given tubular neighborhood of a dynamical leaf, and consider the holonomy map from $\mathbb L_0$ to a small disk transverse to the lamination. Such holonomy maps are quasi-regular, hence critical points are isolated. The critical points are exactly given by the tangencies to $\mathbb L_0$, thus finiteness follows from compactness of $J^+_R$.
\end{proof}

We may assume, by either increasing $R$ or by changing $y_0$, that if a global leaf has more than one tangency with $\mathbb L_0$, then those tangencies are all contained in a single semi-local leaf. This is not necessary for what follows but makes the statement and proof of the Tubular Covering Lemma below more convenient.

\medskip

We first consider a planar covering lemma. Let $K \subset \mathbb C$ be compact, let $\gamma > 1$ and $s \in \mathbb N$. For each $\alpha \in K$ let $\Lambda(\alpha) \subset K$ be a set of order at most $s$, defining an equivalence relation, i.e. $\alpha \in \Lambda(\beta)$ if and only if $\beta = \Lambda(\alpha)$.

We assume that the sets $\Lambda(\alpha)$ vary lower semi-continuously with $\alpha$, i.e.
$$
\Lambda(\alpha) \subset \liminf_{\alpha_j \rightarrow \alpha} \Lambda(\alpha_j).
$$
We assume moreover that
$$
\widehat{\Lambda(\alpha)} = \limsup_{\beta \rightarrow \alpha} \Lambda(\beta)
$$
is also finite and of order at most $s$. Finally, we assume that there exists a Lipschitz constant $C>1$ for the equivalence relation. That is, for $\delta>0$ sufficiently small and $\alpha \in K$, the set
$$
\bigcup_{\beta \in D_{\delta}(\alpha)} \Lambda(\beta)
$$
is contained in at most $s$ disks of radius $C \cdot \delta$. Here we write $D_\delta(\alpha)$ as usual for the disk centered at $\alpha$ of radius $\epsilon$.

We define the sets
$$
E(\alpha,\epsilon) = \bigcup_{\beta \in \Lambda(\alpha)} D_{\gamma \cdot \epsilon}(\beta).
$$

\begin{lemma}[Planar Covering Lemma]
There exists $N$ such that for every sufficiently small $\epsilon>0$ the set $K$ can be covered by a finite collection $\{D_\epsilon(\alpha)\}$, whose set of centers $A$ can be partitioned into subcollections $A_1, \ldots, A_N$, such that for every $j = 1, \ldots , N$ and every $\alpha, \beta \in A_j$, the sets $E_\epsilon(\alpha)$ and $E_\epsilon(\beta)$ are disjoint.
\end{lemma}
\begin{proof}
Recursively construct finite collection $A = \{\alpha\}$ for which the disks $D_\epsilon(\alpha)$ cover $K$, by at each step selecting a center $\alpha$ that is not yet contained in the previous disks. It follows that there is an upper bound, independent of $\epsilon$, on the number of centers $\alpha \in A$ contained in any disk of radius $\epsilon$. For $\gamma \cdot \epsilon < \delta$ it follows from the Lipschitz bound $C$ that given $\alpha \in K$, the set of points $\beta \in K$ for which $E_\epsilon(\beta)$ and $E_\epsilon(\alpha)$ intersect is contained in at most $s^2$ disks of radius $2 C \cdot \gamma \epsilon$. Therefore there is also an upper bound, again independent of $\epsilon$, on the number of centers in $A$ contained in those larger disks.

It follows that for any $\alpha \in A$ the number of centers $\beta \in A$ for which $E_\epsilon(\beta) \cap E_\epsilon(\alpha) \neq \emptyset$ is bounded by a constant $M$ that does not depend on $\epsilon$. Recursively partition $A$ into sets $A_1, \ldots , A_N$, at each step taking a maximal number of the remaining centers $\alpha \in A$ for which the sets $E_\epsilon(\alpha)$ are pairwise disjoint. This process must end in at most $N \le M + 1$ steps.
\end{proof}

We stress that the bound $N$ is allowed to depend on $C$, $\gamma$ and $s$.

\begin{lemma}[Tubular Covering Lemma]\label{lemma:prep3}
Let $R^\prime > R$. Given $\epsilon_1>0$ sufficiently small, there exists an $N \in \mathbb N$ such that for any sufficiently small $\epsilon_2 > 0$ we can cover the dynamical vertical lamination in $\Delta^2_R$ with finitely many tubular neighborhoods $\{T_i(r_i)\}$ centered at dynamical leaves, satisfying the following:
\begin{enumerate}
\item[(i)] The finite set of tubular neighborhoods can be partitioned into $N$ collections $A_1, \ldots, A_{N}$.
\item[(i)] The tubular neighborhoods in $A_1$ have radius $r_i = \epsilon_1$, and all other tubular neighborhoods have radius $r_i = \epsilon_2$.
\item[(ii)] For $\alpha = 1, \ldots , N$ and $T_i(r_i), T_j(r_j) \in A_\alpha$ one has
$$
T_i(k^2 \cdot r_i, R^\prime) \cap T_j(k^2 \cdot r_j, R^\prime) = \emptyset.
$$
\end{enumerate}
\end{lemma}
\begin{proof}
Note that each leaf of the dynamical vertical lamination must pass through the line $\mathbb L_0$, so it is sufficient to consider tubular neighborhoods that cover the intersection of the dynamical vertical lamination with $\mathbb L_0$.

By lemma \ref{lemma:prep2}, there are only finitely many semi-local
leaves with horizontal tangencies in $\mathbb L_0$. We may assume that
$\epsilon_1 >0$ is sufficiently small such that the corresponding
tubular neighborhoods $T_i(k^2 \cdot \epsilon_1, R^\prime)$ do not
intersect. Let $A_1$ be the set of corresponding tubular neighborhoods
$T_i(\epsilon_1)$.
{\em From now on we consider tubular neighborhoods centered at leaves not contained in these finitely many tubular neighborhoods.}

We claim that we are left with the situation of the Planar Covering Lemma. The set $K$ is the intersection of the remaining dynamical vertical lamination with $\mathbb L_0$. The equivalence classes $L(\alpha)$ are given by the intersection points of semi-local dynamical leaves in the bigger bidisk $\Delta^2_{R^\prime}$.

Recall from Lemma \ref{boundeddegrees} that the semi-local leaves are branched covers with uniformly bounded degrees. The lower semi-continuity, and upper bound on $\widehat{L(\alpha)}$ follow as in the proof of Lemma \ref{boundeddegrees}.

We can choose $\epsilon_2>0$ sufficiently small so that for any tubular neighborhood $T_i(\epsilon_2, R)$ not contained in one of the tubular neighborhoods in $A_1$, the intersection $T_i(k^2 \cdot \epsilon_2, R^\prime) \cap \mathbb L_0$ consists of a finite number of connected components, each containing an intersection point of the core leaf.

Each connected component of the intersection closely resembles an ellipse, whose direction and eccentricity is determined by the tangent vector of the leaf at the corresponding intersection point with $\mathbb L_0$. Since we consider only sufficiently small tubular neighborhoods of points bounded away from the tangencies in $\mathbb L_0$, the eccentricity of these ellipses is bounded. In other words, there exists $\ell>1$ independent of $\epsilon_2$ sufficiently small, such that each component of each $T_i(k^2 \cdot \epsilon_2, R^\prime) \cap \mathbb L_0$ is contained in a disk of radius $\ell \cdot \epsilon_2$, and contains the concentric disk of radius $\frac{1}{\ell} \epsilon_2$.

Thus, each intersection $T_i(\epsilon_2) \cap \mathbb L_0$ contains a disk of radius $\frac{1}{\ell} \epsilon_2$ centered at a point $\alpha \in K$, while the intersection of $T_i(k^2 \epsilon_2, R^\prime)$ is contained in a bounded number of disks of radius $\ell k^2 \epsilon_2$, thus we are in the situation of the Planar Covering Lemma for $\gamma = \ell^2k^2$. The existence of the Lipschitz constant $C$ follows from the fact that the holonomy maps are Lipschitz and the bound from below on the angle between the transversal $\mathbb L_0$ and the remaining dynamical vertical lamination.

The existence of the partition $A_1,\ldots, A_N$ therefore follows from the Planar Covering Lemma. The constants $C$ and $\gamma$ depend on $\epsilon_1$, hence so does $N$, but $N$ is independent of $\epsilon_2$.
\end{proof}

\begin{figure}[ht]
\centering
\begin{subfigure}[t]{0.5\textwidth}
\centering
\includegraphics[width=0.9\linewidth]{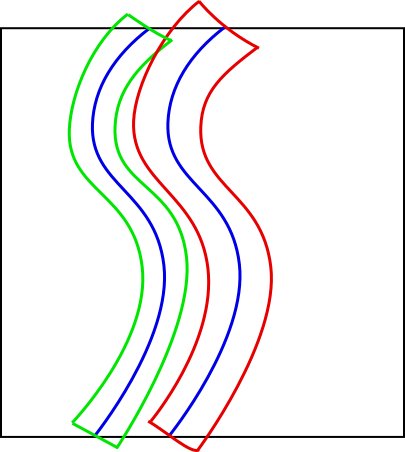}
\caption{Two tubular neighborhoods}
\label{figure:extension2}
\end{subfigure}%
~
\begin{subfigure}[t]{0.5\textwidth}
\centering
\includegraphics[width=0.9\linewidth]{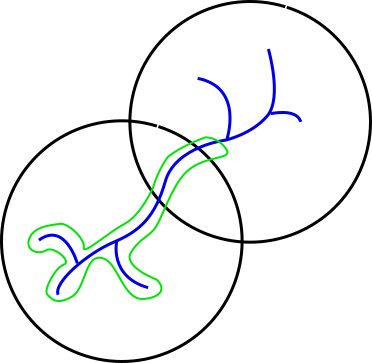}
\caption{From above: thin neighborhood}
\label{figure:extension1}
\end{subfigure}
\label{figure:extension}
\end{figure}

\begin{prop}\label{prop:single}
The dynamical vertical lamination in $\Delta^2_R$ can be extended to an open neighborhood.
\end{prop}
\begin{proof}
We consider $R^\prime>R$ and apply the previous lemma to the dynamical vertical lamination in $\Delta^2_{R^\prime}$. Let $2k^2 \cdot \epsilon_1 < (R^\prime - R)/2$,  let $N$ be as in the previous lemma, and let $\epsilon_2>0$ be sufficiently small such that
$$
2 k^2 \cdot \epsilon_2 \cdot N < (R^\prime - R)/2,
$$
where as before $k$ is an upper bound for the Lipschitz constant of the holonomy maps. We can cover the dynamical vertical lamination in $\Delta^2_R$ by tubular neighborhoods as in the previous lemma, and write $A_1, \ldots , A_{N}$ for the partition into pairwise disjoint tubular neighborhoods. We may assume that $\epsilon_1$ and $\epsilon_2$ are chosen sufficiently small such that the $\lambda$-lemma can be applied to each tubular neighborhood of radius $k\cdot e_i$, for $i = 1,2$.

We first apply the $\lambda$-lemma to each of the tubular neighborhoods $T_i(k^2 \cdot \epsilon_1, R^\prime)$ in $A_1$, keeping only the leaves that intersect the tubular neighborhood of radius $k\cdot \epsilon_1$. Since the tubular neighborhoods of radius $k^2 \cdot e_1$ are pairwise disjoint, it follows that the new leaves are all pairwise disjoint as well.

Note that while the dynamical leaves were global leaves, the new leaves are semi-local, and contained in the tubular neighborhoods in $A_1$.  In order to guarantee that the new leaves are still graphs over the entire cores of the tubular neighborhoods in $A_2, \ldots, A_N$ that they intersect, we reduce the radius of the bidisk $\Delta^2_{R^\prime}$ by $2k^2 \cdot \epsilon_1$.

Note that the laminations constructed using the $\lambda$-lemma may not be Lipschitz. In order to guarantee preserve the modulus of continuity $k$ for each of the selected tubular neighborhoods that will be used in later steps, we keep only the newly constructed leaves that intersect a sufficiently thin neighborhood of the dynamical vertical lamination, see the two tubular neighborhoods illustrated in Figure \ref{figure:extension1}, where the dynamical leaves are represented by the blue continuum, and only the new leaves in the small green neighborhood are kept. Since the holonomy maps will still be continuous, choosing a thin enough neighborhood of the dynamical vertical lamination will still guarantee that the leaves that intersect the tubular neighborhoods of radii $\epsilon_2$ and $k\cdot \epsilon_2$ will still be contained in the corresponding tubular neighborhoods of radii $k \cdot \epsilon_2$ and $k^2 \cdot \epsilon_2$ respectively.

We continue with the tubular neighborhoods in $A_2, A_3, \ldots , A_{N-1}$, and finally to those in $A_N$, each time following the same procedure as above: first apply the $\lambda$-lemma to each tubular neighborhood of radius $k^2 \cdot \epsilon_2$ in the current bidisk, keeping only those leaves that intersect the tubular neighborhood of radius $k \cdot \epsilon_2$, then decrease the radius of the bidisk by $k^2 \cdot \epsilon_2$, and finally only keeping the newly constructed leaves in a very thin neighborhood of the dynamical vertical lamination in order to maintain the modulus of continuity.

By our choice of $\epsilon_2$ we end up with the required extended lamination on a bidisk of radius at least $R$.
\end{proof}

We will refer to this extension of the dynamical vertical lamination
as the \emph{artificial vertical lamination}, and denote it by
$\mathcal{L}$. By slight abuse of terminology, we will also write
$\mathcal{L}$ for the union of the vertical leaves, which gives a
neighborhood of $J^+_R$. We may assume that the dynamical vertical
lamination is extended to a thin enough neighborhood such that
$\mathcal{L}$ is contained in the region where the dominated splitting
is defined, and by continuity we may assume that its tangent bundle
lies in the vertical cone field.
{\em From now on we write $\mathcal{N}(J^+_R)$ for the region where both the cone field and the artificial vertical lamination are defined.}

\subsection{Adjusting the artificial vertical lamination on wandering domains}
\label{adjustment}

Note that there is no reason for the artificial vertical lamination $\mathcal{L}$
to be invariant under $f$. Here we discuss how to define and modify
the lamination on the (hypothetical) wandering Fatou components.

\begin{figure}[ht]
\centering
\includegraphics[width=0.6\linewidth]{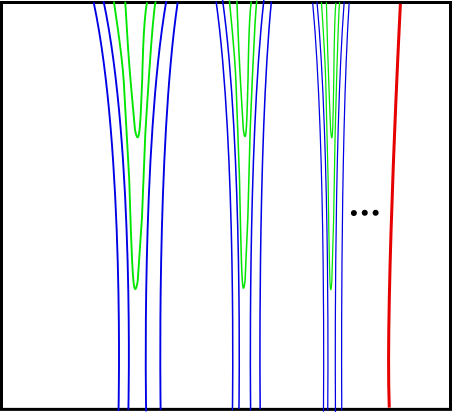}
\caption{\label{wandering4} The pullback is not continuous in the limit}
\end{figure}

Recall the set $\VV^+$ (\ref{VV}) foliated by stable manifolds $W^s(z)$.
Let $U$ be a wandering Fatou component of $f$.
As $ f^n z \to J$ for any $z\in U$,  we have: $U\subset \VV^+$.
So, $U$ is foliated by strong stable manifolds; we call it
the  {\em dynamical foliation}  $\FF_U$  of $U$.
Putting these foliations together, we obtain the  invariant dynamical foliation
on the union $\UU$  of all wandering components.

However, in general, this foliation
cannot be  extended to the closure of this union (see Figure \ref{wandering4}).
To deal with this problem, we combine this dynamical foliation on
some ``semi-local'' wandering components  with
the non-dynamical  extension on the others, to obtain
{\em a lamination of $\UU\cap \Delta^2_R$ which is invariant everywhere
except finitely many semi-local wandering components.}


\begin{figure}[t!]
\centering
\begin{subfigure}[t]{0.5\textwidth}
\centering
\includegraphics[width=0.9\linewidth]{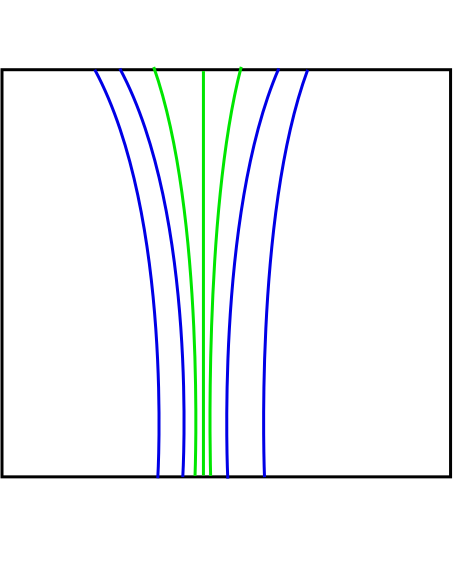}
\caption{Artificial vertical lamination}
\end{subfigure}%
~
\begin{subfigure}[t]{0.5\textwidth}
\centering
\includegraphics[width=0.9\linewidth]{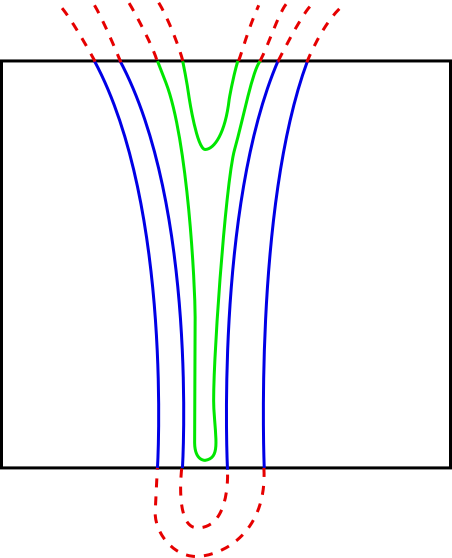}
\caption{Dynamical vertical lamination}
\end{subfigure}
\caption{\label{wandering} Conflicting laminations in a wandering component}
\end{figure}

\medskip

 A {\em semi-local wandering component}  $V$ is a
connected component of $U \cap \Delta^2_R$.
Note that there are at most finitely many semi-local wandering components $V \subset U \cap \Delta^2_R$  not
contained in $\mathcal{N}(J^+_R)$.
We refer to such a component $V$ as a \emph{component with hole}.

Let $V$ be a semi-local component with hole, let $V^{-1}$ be a connected
component of $f^{-1}(V) \cap \Delta^2_R$, and assume that neither
$V^{-1}$ nor any connected component of $f^{-n}(V^{-1}) \cap
\Delta^2_R$ has a hole. Then $V^{-1}$ is contained in $\mathcal{L}$
and hence foliated by vertical leaves. Note that vertical leaves in
$\mathcal{L}$ sufficiently close to the (vertical) boundary of
$V^{-1}$ are necessarily \emph{dynamical} leaves, and recall that the
dynamical vertical lamination is invariant under $f$. We modify the vertical
lamination $\mathcal{L}$ by pulling back the leaves in $V^{-1}$ to all
components of $f^{-n}(V^{-1})\cap \Delta^2_R$ for all $n \ge 2$.

If $V^{n}$ lies in $\mathcal{N}(J^+_R)$ for all $n \ge n_0$, then the artificial vertical lamination on $V^{n_0}$ is dynamical, and we can pull back the \emph{dynamical} lamination on $V^{n_0}$ to all components $V^j$ with $0 \le j \le n_0$. See Figure \ref{wandering} for a sketch of the two conflicting laminations that one obtains by pulling back the dynamical vertical lamination to $V^{-1}$. Note that the artificial vertical lamination near the boundary of the component is dynamical, and is therefore identical in both pictures.

By following this procedure for all grand orbits of Fatou components with holes, we obtain a lamination that is invariant on all but finitely many components, and for each bi-infinite orbit of components there is at most one step in which the lamination is not invariant. To be more precise, if ${V^n}_{n \in \mathbb Z}$ is a sequence of semi-local wandering components with $f(V^n) \subset V^{n+1}$, then there is at most one $n \in \mathbb Z$ for which the leaves of the artificial vertical lamination in $V^n$ are not mapped into leaves of the lamination in $V^{n+1}$, and this can only occur when $V^j$ lies in the region of dominated splitting for $j \le n$ but not for $j = n+1$.


\subsection{Choice of horizontal line}

In the    one-dimensional
argument we considered iterated inverse images of
a given disk $D_r(z)$.
This is problematic in the H\'enon case.
The reason is that such preimages are very likely to land at least partially outside of the bidisk $\Delta^2_R$.
 Instead, we will start with a flat horizontal complex line $\mathbb L_0 $,
map it forward by $f^n$, consider a small disk $V^0$ inside $f^n(\mathbb L_0) \cap \Delta^2_R$,
 and consider the pullbacks $V^1, \ldots V^n$ of this disk.
Of course instead of working with a full horizontal line $\mathbb L_0$ it is equivalent to work with a disk of radius $R$.

We will now make a suitable choice for the horizontal line:

\begin{lemma}\label{good hor line}
There exists $y_0 \in \mathbb C$  with $|y_0| < R$ such  that the
artificial vertical lamination of
$J^+$ is transverse to $\{y = y_0\}$.
\end{lemma}
\begin{proof}
For each horizontal complex plane, the tangencies
of this plane with the artificial vertical lamination are isolated, thus, by
restricting the neighborhood of $J^+$ if necessary, there are at most
finitely many tangencies. We can remove the tangencies one by one by
making arbitrarily small perturbations for which the tangencies are
transferred to nearby leaves in the Fatou set. More precisely, if a
leaf is tangent to a horizontal plane, then each nearby vertical leaf
is tangent to some nearby horizontal plane. Locally the number of
tangencies, counted with multiplicities, is constant.
Thus, we can take any 
nearby leaf in the Fatou set, take the horizontal plane for which that
leaf is tangent, and reduce the number of tangencies in $J^+$ by at
least one.
 After a finite number of perturbations we obtain
a desired horizontal  line  $\{y = y_0\}$.
\end{proof}

\begin{defn}\label{ellnot}[\emph{choice of $\mathbb L_0$}]
From now on we fix $y_0$ so that the dynamical lamination of
$J^+$ is transverse to $\mathbb L_0$, and so that $\mathbb L_0$
does not contain any  parabolic
periodic points.

It follows that the line $\mathbb L_0$ is transverse to the artificial vertical lamination in a sufficiently small neighborhood of $J^+$.
\end{defn}

\comm{
In what follows,
we will consider small "protected" holomorphic disks in $f^n \mathbb L_0 \cap
\Delta_R^2$ and pull them back to $\mathbb L_0$. 
 We will show that the diameters of these  pullbacks
stay bounded, and at the same time obtain a bound on their degrees.
Here the \emph{degree} means the maximal number of intersections
with semi-local leaves of the artificial vertical lamination,
while the \emph{diameter} refers to the Kobayashi diameter in the domain
$\Om$ constructed above,
or some equivalent notion of diameter.
As in one variable, bounds on the diameter will imply bounds on
the degree, and vice versa.
}

\section{Uniformization of wandering components}

In this section we will show that any wandering component $U$ can be
uniformized by the straight cylinder $\D\times \C$ in such a way that
the dynamical foliation of $U$ becomes vertical.
It will imply a bound for the (appropriately understood) degrees
 of the maps $f^n |\, U$.

\subsection{Contraction}

\begin{lemma}\label{shrinking}
    For any wandering component $U$,
the derivatives $\|d f^n \| $ converge to $0$ uniformly on compact
subsets of $U$.
\end{lemma}

\begin{proof}
Replacing $U$ with its iterated image, if needed,
we can ensure that all the images $U^n= f^n(U)$
lie in the domain of dominate splitting. Hence $U$ is filled with
global strong stable manifolds $W^s(z)$.

Arguing by contradiction,  we can find a sequence of unit vectors $v^m \in E^c(z^m )$
converging to a vector $v \in E^c(z)$, $z\in U^0$, and a sequence of
moments $n_m\to \infty$ such that
\be\label{lower bound}
   \| df^{n_m} (v^m) \| \geq \epsi >0.
\ee
Take a small horizontal disk $D\ni z$ tangent to $v$, and
 find a sequence of horizontal disks $D^m \ni z^m$ tangent to the
 $v^m$ and converging to $D$. Since the family of iterates is normal
 near $z$, the derivatives $d(f^n| \, D_m) $  have a uniformly
 bounded distortion. Together with (\ref{lower bound}),   this  implies
$$
       \| df^{n_m} (w) \|\geq \epsi'>0, \quad  \forall\  w\in T D^m,
$$
so the images  $f^{n_m } (D^m) $ are horizontal disks of definite
size.  Hence the local stable manifolds through each of them  fill a ball of definite
radius.
This contradicts the fact that these infinitely many balls must be disjoint yet bounded.
\end{proof}


Recall the set $\VV^+$ (\ref{VV}) foliated by the strong stable manifolds
$W^s (z) $.
For $z \in \VV^+$, let $\phi_z : \C \ra W^s(z)$ be a uniformization of $W^s(z)$, normalized so that $\phi_z (0) = z$ and $ \|d\phi_z (0)\| = 1$. We note that $\phi_z$ is unique up to multiplication in $\mathbb C$ by a constant $e^{i \theta}$.
Hence for  $\zeta = \phi_z(u) \in W^s(z) $ we can define the (asymmetric)  \emph{intrinsic distance} as
$$
     \dist^i (z, \zeta) = |u|,
$$
which is independent of the choice of $\phi_z$.

\begin{lemma}\label{intrinsic vs asymp}
{\rm (i)} For $\zeta\in W^s_\loc (z)$, we have:
$
\dist^i (z, \zeta) \asymp \| z - \zeta\|;
$

\ssk\nin
{\rm (ii)}   There exists an $\epsi>0$  with the following property:
 For  any $\zeta\in W^s(z) \sm W^s_\loc (z) $, there exists
$n\asymp 1+ \log^+   (  \epsi^{-1}  \dist^i (z, \zeta) ) $ such that
$
        \| f^n z - f^n \zeta \| \geq \epsi.
$
\end{lemma}

\begin{proof} (i)
The linearizing maps $\phi_z$ are locally bi-Lipschitz with a
constant continuously depending on $z$, which implies the first assertion.

\msk (ii)
Let us select
$\epsi >0 $ so that each $\phi_z$ maps the disk $ \De_\epsi$ into $W^s_\loc(z)$.
Since
\be\label{uniformlization rule}
f^n \circ \phi_z(u)  = \phi_{z_n}(\la_n u)  \quad
{\mathrm {with}}\  \la_n = \| df^n_z |\, E^s \| ,
\ee
the intrinsic distance is contracted at an exponential rate.
Hence the number of iterates it takes for it to become of order $\epsilon$ depends logarithmically on $\mathrm{dist}^i(z,\zeta)$.
Application of (i) concludes the proof.
\end{proof}

\subsection{Global transversals}

Recall from \S \ref{adjustment} that
given a wandering component $U$,
$\FF_U$ stands for the dynamical vertical lamination of  $U$
by the global stable manifolds $W^s (z) \isom \C$.

Let us say that $D$ is a \emph{global transversal} to a
wandering component $U$ if

\ssk\nin (T1)
 $D$ is  a non-singular holomorphic disk
properly embedded into $U$;

\ssk\nin (T2)
$D$ is relatively compactly contained in a non-singular holomorphic
curve $D'$;

\ssk\nin (T3)
 For any $z\in \bar D$, the curve  $D'$ is transverse to   the stable
 line $E^s_z$ (see \S \ref{stable manifolds}).

\begin{lemma}\label{horizontality}
  If $D$ is a global transversal to a
wandering component $U$,
then for $n$ large enough the images $f^n (D)$ are horizontal with
respect to the cone field.
\end{lemma}

\begin{proof}
By assumption (T3),
 the angle between the tangent line  $T_z D$ and the stable line $E^s_z$
is bounded below by some $\alpha>0$ independent of $z\in D$.
Hence it takes a bounded amount of iterates  to bring $T_z D$ to a
horizontal cone.
\end{proof}

\begin{lemma}\label{existence}
Let $U$ be a wandering component, let $z \in U \cap \Delta^2_R$. Then $W^s(z) \subset U$ intersects a global transversal.
\end{lemma}
In particular,  any wandering domain contains a global transversal.
\begin{proof}
By Lemma \ref{good hor line}, there exists a horizontal line $\mathbb L_0$
transverse to the dynamical vertical lamination on $J^+$. For large $n \in \mathbb N$ connected components of $f^n(U) \cap \Delta^2_R$ will be contained in an arbitrarily small neighborhood of $J^+$, and hence the dynamical vertical lamination in those components is transverse to $\mathbb L_0$. Since $f$ is substantially dissipative, the semi-local strong stable manifold $w^s_R(f^n(z))$ intersects $\mathbb L_0$ in a connected component $D_n \subset \mathbb L_0 \cap f^n(U)$.

Since $f^n(U)$ is a Fatou component, the maximum principle implies that $D_n$ is simply connected. Let $D_n^\prime \subset \mathbb L_0$ be a slightly larger domain where the dynamical vertical lamination is still transverse to $\mathbb L_0$. Pulling back by $f^n$ gives the required $D \subset D^\prime$.
\end{proof}

Select a continuous unit vector field $v(z)\in E^s_z$ on $D^\prime$, and for $z\in D$,
let  $\phi_z : \C \ra W^s(z)$ be the uniformization of  $W^s(z)$ normalized so that $\phi_z (0) = z$ and $\phi^\prime_z  (0) = v(z) $.
Then we obtain a continuous map
\be\label{Phi}
\Phi:  D \times \C \ra  U, \quad ( z, u ) \mapsto \phi_z (u).
\ee
Our goal is to prove that $\Phi$ is a homeomorphism.

\subsection{Holonomy group}

Let $D$ and $\Delta$ be global transversals to $U$, let $z \in D_1$, and suppose that the global stable manifold $W^s(z)$
intersects $D_2$ in point $w \Delta$. Then holonomy induces a
map $h$ from a neighborhood of $z$ in $D$ to a neighborhood of $w$ in $\Delta$.

\begin{lemma}\label{holonomy}
The map $h$ admits a unique extension along any path
$\gamma$ in $D_1$.
\end{lemma}

\begin{proof}
Assume there exists a path  $\gamma: [0, 1]  \ra D$,
$z^0 = \gamma(0) $,
such that  $h$ extends along $\gamma : [0,1)\ra D$   but does not extend to
$z^1 = \gamma(1)\in D$.
We let
$$
z^t = \gamma(t), \quad  h(z^t ) = \Phi (z^t, u^t),  \quad  t\in [0,1).
$$
Since $h$ does not extend to $z^1$ it follows that
$$
  \dist^i (z^t  , h( z^t  ) ) = |u^t|  \to \infty\
    {\mathrm {as}} \ t\to 1.
$$
Otherwise we would have a subsequence $t_k\to 1$
with bounded $u^{t_k} $. Then we could take a limit point $u^1$
of the $u^{t_k}$ and obtain a local holonomy from $z^1$ to
$h(z^1) = \Phi (z^1, u^1)$. This local holonomy must then agree with the holonomy along $\gamma$ for $t$ close to $1$, giving a contradiction.

For $t\in [0,1)$, let $\de^t$ be the ``intrinsically straight''
 path in $W^s (z^t) $ connecting
$z^t$ with  $h(z^t)$:
$$
       \de^t : [ 0, u^t]  \ra W^s (z^t ), \quad \de^t (\cdot) =
       \Phi(z^t , \cdot).
$$
Let $D_n= f^n(D)$,  $\gamma_n= f^n(\gamma)$, $z^t_n = f^n(z^t)$, $\de^t_n = f^n (\de^t)$,
and let $h_n=  f^n \circ h \circ f^{-n}$ be the push-forward holonomy
defined on $\gamma_n: [0, 1) \ra D_n $.
By  Lemma \ref{shrinking},
$$
   \length \gamma_n \to 0 \ {\mathrm{and}}\
     \length \de^0_n   \to 0 \
   {\mathrm{  as}}\  n\to \infty,
$$
where $\length $ stands for the Euclidean length.
In particular, the path $\de_n^0 $ lies in $W^s(z^0_n)$ for
$n$ sufficiently big. On the other hand, there exists an $\epsi>0$ such
that for any given $n$  and $t\in [1-\eta_n, 1) $ with $\eta_n\to 0$,
$\length \de^t_n  \geq  \epsi  $.  Let us select the smallest $t= t_n$ with this
property.

Let us use a local coordinate system near $z_n^0$
with axes $E^c$ and $E^s $ at that point.
 Then the holonomy $h_n$ on the short  path  $\gamma_n (t)  ) $, $0\leq t\leq t_n$,
quickly goes from  a small height (equal to $ \length \de_n^0$) to a
definite height (of order $\epsi$),
so it has a big average slope. It follows that somewhere either $\Delta_n$ or $D_n$
 must have a  small angle with the stable direction, contradicting the property that for large $n$
they are horizontal with respect to the cone field.
\end{proof}

\begin{corollary}\label{globalholonomy}
Any local holonomy map $h$ extends to a  homeomorphism $D\ra \Delta$.
\end{corollary}
\begin{proof}
Since $D$ is simply connected and any $h$ extends uniquely along all paths, the usual argument of the Monodromy Theorem implies that
$h$ extends uniquely to a global continuous map $D\ra \Delta$.
Since the same is true for $h^{-1}$, it is a homeomorphism.
\end{proof}

In particular, we note that when $\Delta = D$, the holonomy maps form a group $\HH_D$ of homeomorphisms of $D$.

Denote by $W^s(D)$ the union of the strong stable manifolds through $D$.

\begin{lemma}
Let $D$ be a global transversal to a
 wandering component $U$.
Then  $U = W^s(D)$.
\end{lemma}
\begin{proof}
It is immediate that $W^s(D)$ is an open subset of $U$. Let $z \in U \cap \overline{W^s(D)}$. Let $n \in \mathbb N$ such that $f^n(U)\cap \Delta^2_R$ is contained in $\VV^+$. By Lemma \ref{existence} there exists a global transversal $\Delta \subset L$ intersecting $W^s_R(z_n)$, or equivalently $z_n$ is contained in the semi-local strong stable manifold of some $w\in \Delta$. Then $W^s(\Delta)$ contains a neighborhood of $z_n$, and hence intersects $f^n(D)$. Thus, we obtain a local holonomy map from $\Delta$ to $f^n(D)$, which by Corollary \ref{globalholonomy} extends to a homeomorphism $h:\Delta \rightarrow D_n$. It follows that $W^s(D_n) = W^s(\Delta)$, implying $z_n \in W^s(D_n)$. We conclude that $z \in W^s(D)$, completing the proof.
\end{proof}

\subsection{Uniformization}

\begin{prop}
Let $D$ be a global transversal to a
wandering component $U$.
Then any stable manifold intersects $D$ in at most one point.
\end{prop}

\begin{proof}
Suppose for the purpose of a contradiction that a strong stable
manifold intersects $D$ in two distinct points,
and denote the induced holonomy homeomorphism  by $h: D\rightarrow D$.
Let us consider push-forward holonomies
$$
   h_n= f^n\circ h\circ f^{-n} : D_n \ra D_n, \quad
     {\mathrm{where}}\ D_n= f^n(D).
$$
For  any point $z\in D$ and $n$ big enough (depending on $z$),
it is the holonomy along the local stable foliation in some flow box $B_i$  containing $z_n=f^nz$.
By the $\lambda$-lemma \cite{MSS},
$h_n$  is locally quasiconformal (``qc'') near $z_n$.
Since a biholomorphic map $f^n$  does not change the dilation,
$h$ is locally qc near $z$, with the same local dilatation
(depending only on $B_i$ but not on $h$, $z$  and $n$). Since there are only finitely many flow boxes $B_i$
covering the whole domain of the dominated splitting,
these local dilatations are uniformly bounded for all $h\in \HH\equiv \HH_D$ and  $z\in D$.
Hence each $h\in \HH$ is globally qc on $D$ with uniformly bounded
dilatation, so the holonomy group $\HH$  acts uniformly qc on $D$.

Furthermore,  $\HH$ acts freely on $D$
since fixed points of the action would be tangencies between the
stable foliation and $D$.

Moreover, for any point $z \in D$,
the intersection $W^s(z) \cap D$ is discrete in the intrinsic topology
of $W^s(z)$. Otherwise, there would exist distinct  points
$w^m = \phi_z  (u^m )  \in D $ with bounded $u^m$.
Then we could  select a subsequence converging to a point
$w  = \phi_z (u)\in \overline D \subset D' $,
which would be a non-isolated point of the intersection
$W^s(z) \cap D'$.

In fact, this discreteness is
uniform in the following sense:
For any $M$ and any $z\in \overline D$,
$\dist^i (z, hz) > M$ for all but finitely many holonomy
homeomorphisms $ h\in \HH$.
Indeed, if there is
sequence $w^m =  h_m (z^m) = \Phi(z^m , u^m)$ with bounded $u^m$,
then we can select a converging subsequence
$u^m \to u $,  $z^m\to z\in\overline D$, $w^m\to w\in \overline D$ so that
$w= \Phi(z, u) = h(z)$ for some $h\in \HH$.
Then $h(z^m ) = \Phi(z^m , t^m) $ with $t^m \to u$, and hence
$|t^m- u^m|\to 0$.
But for $m$ big enough, both $\Phi(z^m , t^m) $  and $w^m= \Phi(z^m ,u^m)$
lie in the same flow box $B$ around $w$.
It follows that they lie in the same local leaf of $B$.
Since the latter intersects $\overline D$ at  a single point,
we conclude that $w^m=  h(z^m)$,
and hence $h_m=h$ (for all big enough $m$).

Let us now show that $\HH$  acts properly discontinuously on $D$,
i.e, for any two  neighborhoods $Z$ and $W$
compactly contained in $D$, we have $h(Z)\cap W =\emptyset$ for all but finitely
many $h\in \HH$. Indeed, assume there is  a sequence of distinct
$ h^m \in \HH$
and of points $z^m \in Z$, $w^m = h^m (z^m) \in W$.
As we have just shown,  $\dist^i (z^m , w^m) \to \infty$.
Now we can apply our usual argument to arrive to a contradiction.
Namely, there exist moments $n_m\to \infty$ that bring the  points $z^m$ and
$w^m$ to the same local stable manifold, implying that
 $ \| f^{n_m} (z^m) - f^{n_m} ( w^m) \| \asymp 1$,
which contradicts to Lemma~\ref{shrinking}.

Hence the quotient $S= D/\HH_D$ is a qc surface
(i.e., a surface endowed with qc local charts with uniformly bounded dilatation).
Taking any conformal structure (a Beltrami differential) $\mu$ on $S$
and pulling it back to $D$, we obtain an $\HH$-invariant conformal structure
on $D$. By the Measurable Riemann Mapping Theorem,
there exists a qc map $\psi: D \ra \D$ such that
$g= \psi\circ h \circ \psi^{-1}$ is  M\"obius for any $h\in \HH_D$.

Let $\zeta\in \mathbb D$ and denote its orbit by $\zeta_n = g^n(\zeta)$.
Then   the hyperbolic distance between $\zeta_n$ and $\zeta_{n+1}$
is independent of $n$ since it  is preserved under holomorphic automorphisms.
Let  $z_n = \psi^{-1}(\zeta_n) \in D$.
Since $\psi$ is quasiconformal, it is a quasi-isometry, that is,
$\psi$ expands the hyperbolic distance by a bounded factor for scales bounded away from zero.
 It follows that the hyperbolic distance between $z_n$ and $z_{n+1}$ is bounded for all $n$.

Since $g$ does not have fixed point in $\D$,
 the $\zeta_n$ converge to a Denjoy-Wolff point in $\partial\mathbb  D$.
Hence the sequence $(z_n)_{n \in \mathbb N}$ escapes to the boundary $\di D$.
Since near the boundary the hyperbolic metric of $D$  explodes relatively the
Euclidean metric of $D'$, we conclude that
\be\label{dist goes to 0}
\| z_n - z_{n+1} \|\to 0.
\ee

Note that all the points $z_n$ lie in the same stable manifolds
$W^s(z)$,
so we can measure the intrinsic distance  between them.
Property (\ref{dist goes to 0}), together with   Lemma~\ref{intrinsic vs asymp},
imply that  $\dist^i (z_n, z_{n+1}) \to 0$. It follows that any limit
point $ q \in\overline D\subset  D' $  for $(z_n)$ is a tangency between
$D'$ and $ W^s(q)$, and this contradiction completes the proof.
\comm{***
Now consider a subsequence $y_{n_j}$ converging to $p \in \partial W$. Then it follows that the subsequence $y_{n_j+1}$ converges to $p$ as well. Since the inner radius of $f^k(W)$ converges to zero as $k \rightarrow \infty$, it follows that
$$
\|f^k(y_{n_j}) - f^k(y_{n_j+1})\| < \epsilon
$$
for all $j \ge 1$ and all $k$ sufficiently large. Lemma \ref{degree-lemma1} therefore implies that for $k$ sufficiently large the pairs $f^k(y_{n_j})$ and $f^k(y_{n_j+1})$ must lie in local stable manifolds for all $j$. Fix such $k$. As $f^k(y_{n_j})$ and $f^k(y_{n_j+1})$ approach $f^k(p)$ as $j \rightarrow \infty$, and these pairs lie on local stable manifolds, we can pull back by $f^k$ and conclude that $y_{n_j}$ and $y_{n_j+1}$ lie on local stable manifolds for $j$ sufficiently large. This is in contradiction with our earlier observation that  $\overline{W}$ is transverse to the dynamical vertical lamination, and the proof is complete.
********************}
\end{proof}

\begin{remark}
  The above application of the Measurable Riemann Mapping Theorem is a
  special case  of Sullivan's Theorem concerning  uniformly qc group
  actions \cite{S2}, \cite{Tukia}.
\end{remark}

Since stable manifolds in $U^0$ intersect $D$ in a unique point,
we conclude:

\begin{corollary}\label{global straightening}
Let $D$ be a global transversal to a
wandering component $U$.
  Then the  uniformization $\Phi: D\times \C \ra U$ is a vertically
  holomorphic  homeomorphism.
\end{corollary}

\subsection{Degree bound}

Let $U^0$ be a semi-local wandering component of a wandering
component $U$, and let $U^n$ be the component of $f^n(U)\cap \Delta^2_R$ containing
$f^n (U^0)$. we define the degree of $f^n : U^0 \ra U^n$
as the maximal number of semi-local dynamical leaves in
$U^0 $ that are mapped into a single semi-local dynamical leaf in $U^n $.

\begin{lemma}\label{lemma:forwarddegree}
For any semi-local wandering component $U^0$,
the degree of $f^n: U^0 \rightarrow U^n$ is uniformly bounded over all
$n$ (with a bound depending on $U^0$).
\end{lemma}
\begin{proof}
By replacing $U^0$ with an appropriate $U^m$,
we can ensure that all the domains $U^n$, $n\geq 0$,
are contained in the neighborhood $\NN(J^+)$,
so the dynamical and the artificial vertical laminations coincide on these domains. Note that degrees of compositions are sub-multiplicative, so a bound on the degrees of the maps $f^n: U^m \rightarrow U^{m+n}$ implies a bound on the degrees of the maps $f^{m+n}: U^0 \rightarrow U^{m+n}$.

Let us consider  a horizontal line $L= \mathbb L_0$,
and let $t$ be the number of tangencies between $\mathbb L_0$ and the artificial vertical lamination.
Let $ L^0 := U^0\cap L$, and let $L^n := f^n(L^0)$.

As each semi-local dynamical leaf of  $U^0$ intersects $L^0$,
the degree  of $f^n : U^0 \ra U^n$ is bounded by the maximal number of
intersections between $ L^n $ and the vertical leaves of $U^n$.
Since  the artificial vertical lamination of $U^n$ coincides with the
(invariant) dynamical vertical lamination,
the degree of $f^n : U^0 \ra U^n$ is bounded by the maximal number of
intersections between $ L^n $ and dynamical leaves of $f^n(U)$.

Let $D^n$ be a global transversal to $f^n (U)$ and let
$\Phi_n: D^n \times \C \ra f^n(U)$ be the corresponding uniformization.
Then the maximal number of
intersections between $ L^n $ and dynamical leaves of $f^n(U)$
is equal to the degree $d$ of the horizontal projection $\Phi_n^{-1} (L^n)
\ra D^n$. This projection is a branched covering since $L_n$ is
properly embedded into $U^n$.
By the  Riemann-Hurwitz  formula, $d$ equals at least one plus the number of tangencies  (counted with multiplicities)
between $L_n$ and  the dynamical foliation.
But the latter is preserved by the dynamics, so it equals the
number of tangencies between $L^0$ and the dynamical foliation of $U^0$,
which is bounded by $t$. The conclusion follows.
\end{proof}

\section{Horizontal lamination and $\deg_\crit$}

We let $U^0$ be a semi-local wandering Fatou component, and for $n \in
\mathbb N$ we write $U^{-n}$ for a semi-local wandering components
satisfying $f^n(U^{-n}) \subset U^0$.

\medskip
{\bf Assumption A.}
 {\em Let us say that a semi-local wandering component $U^0$ satisfies Assumption A
if  all possible choices of the domains $U^{-n}$, $n\geq 0$,
are contained in $\mathcal{N} ( J^+_R) $ (defined at the end of \S\ref{loc and glob extensions}),
and thus in the domain of dominated splitting.}

\smallskip
In what follows, through Corollary \ref{bound for A},
we will assume that $U^0$ satisfies Assumption $A$.

\medskip
We write $W^{-n}$ for a connected component of $U^{-n} \cap \mathbb
L_0$, and consider holomorphic disks $f^n(W^{-n}) \subset U^0$,
ranging over all $n \in \mathbb N$ and all choices of $U^{-n}$ and $W^{-n}$.
Since the horizontal line $\mathbb L_0$ is chosen so that it is
transverse to the artificial vertical lamination near $J^+$,
for $n$ sufficiently large the tangent spaces to the holomorphic disks
$f^n(W^{-n})$ are contained in the horizontal cone field.
 In fact, by taking $n$ sufficiently large we may assume that the horizontal cone field is arbitrarily thin.

Let $z \in U^0$, and consider small bidisks $\Delta^2_\rho(z) \subset \Delta^2_r(z) \subset U^0$, with respect to affine coordinates contained in the horizontal respectively vertical cone field, and with boundary bounded away from $J^+$. For $\rho \ll r \ll 1$  any connected component of $f^n(W^{-n}) \cap \Delta^2_r(z)$ intersecting $\Delta^2_\rho(z)$ is a horizontal graph. The collection of these graphs form a normal family, hence any sequence has a subsequence that converges locally uniformly. We consider the Riemann surfaces $\mathcal{S}_\nu$ that are locally given as uniform limits of these horizontal graphs.

\begin{lemma}
If two limits $\mathcal{S}_1$ and $\mathcal{S}_2$ intersect, then they are equal.
\end{lemma}
\begin{proof}
Assume for the purpose of a contradiction that $\mathcal{S}_1$ and $\mathcal{S}_2$ intersect in a point $z_0$, but that they locally do not coincide. Let us assume that $\mathcal{S}_1$ and $\mathcal{S}_2$ are locally given as limits of graph in respectively $f^{n_j}(W^{-n_j})$ and $f^{m_l}(W^{-m_l})$. In local coordinates $\Delta^2_R(z_0)$ we can write $\mathcal{S}_1 = \{y = \varphi(x)\}$ and $\mathcal{S}_1 = \{y = \psi(x)\}$.

It follows that for large enough $j$ and $l$ the horizontal graphs in $f^{n_j}(W^{-n_j})$ and $f^{m_l}(W^{-m_l})$ intersect at a point near $z_0$.

Let $j$ and $l$ be large with, say, $n_j$ larger than $m_l$. Then it follows that there exists a point
$$
\zeta_{m_l} \in f^{n_j - m_l}(W^{-n_j}) \cap W^{-m_l}
$$
for which $f^{m_l}(\zeta_{m_l})$ lies near $z_0$. As discussed above,
for some large fixed $N$ independent of $l$ and $j$ the local graphs
in $f^N\circ f^{n_j - m_l}(W^{-n_j})$ and $f^N(W^{-m_l})$ are both horizontal.
 Given that the map $f$ acts as an exponential contraction in the  vertical direction,
while being at most sub-exponentially contracting in the horizontal direction (by Lemma \ref{newlemma4}),
it follows that near the point $f^{m_l}(\zeta_{m_l})$ the distance
between the horizontal graphs in $f^{n_j}(W^{-n_j})$ and
$f^{m_l}(W^{-m_l})$ shrinks exponentially fast as $l \rightarrow \infty$.
Therefore,  the limits $\varphi(\mathbb D_r)$ and $\psi(\mathbb
D_{r^\prime})$ coincide locally. Since they are both proper holomorphic disks, they must coincide globally, which gives a contradiction.
\end{proof}

We will refer to the collection of Riemann surfaces $\mathcal{S}_\nu$
as the \emph{horizontal lamination} in $U^0$.
It is clear that this lamination is contained in the backward Julia set $J^-$.
Since $U^0$ is Kobayashi hyperbolic, each leaf is a hyperbolic Riemann
surface.
The leaves are locally given as limits of horizontal graphs, hence are themselves also horizontal.



\begin{lemma}
The horizontal and the vertical laminations do not share leaves.
\end{lemma}

\begin{proof}
Indeed, vertical leaves intersect the boundary of $\Delta^2_R$,
while  horizontal do not.
\end{proof}

\begin{corollary}
For any semi-local component $U^0$ satisfying Assumption (A)
   the order of tangencies between horizontal and vertical leaves in
   $U^0$  is bounded.
\end{corollary}

\begin{proof}
It follows from two observations:

\ssk \nin  $\bullet$ The order of tangencies between the leaves depends  upper
  semi-continuous on the intersection  point.

\ssk \nin  $\bullet$
Near $J^+$ the laminations are transverse.
\end{proof}

\begin{corollary}
For any semi-local component $U^0$ satisfying Assumption (A),
any preimage $U^{-j}$, and any  component   $W^{-j}$  of the horizontal slice $\bL_0 \cap U^{-j}$,
the orders of tangency of the holomorphic disk $f^j ( W^{-j} )$ with the
dynamical vertical foliation $\FF_{U^0}$ is bounded.
\end{corollary}

\begin{proof}
   It is obvious for small $j$.
For a large $j$, the disk $f^j ( W^{-j} )$ is a small perturbation of
some horizontal leaf $L$ in $U^0$, so the order of its tangencies
between  $f^j ( W^{-j} )$  and $\FF_{U^0}$ is bounded by the order of
tangencies between $L$ and $\FF_{U^0}$ .
\end{proof}

For a component $U^0$ satisfying Assumption (A),
let us define  $\deg_\crit (U^0)$  as the maximum of the   order of
tangency between the above holomorphic disks $f^j ( W^{-j} )$
and the dynamical vertical foliation $\FF_{U^0}$.

\begin{corollary}\label{bound for A}
$\deg_\crit (U^0)$ is bounded over all components $U^{0}$
satisfying Assumption A.
\end{corollary}

\begin{proof}
In the case where all forward components $U^{n}$
are contained in $\mathcal{N}(J^+)$
then the dynamical vertical lamination $\FF_{U^0}$
is tangent to the vertical line field $E^v$, which is  transverse to
the horizontal lamination. Hence $\deg_\crit U^0 = 1$ in this
case.

Therefore we only need to consider the case where some forward component $U^n$ is not contained in $\mathcal{N}(J^+)$. Since $\deg_\crit (U^0)$ is defined by means of two dynamical laminations, both invariant under $f$,
it remains the same for all semi-local preimages $U^{-j}$ of $U^{0}$.
Thus, it suffices to consider only those semi-local components $U^n$
for which $U^{n+1}$ does not satisfy Assumption A. Since there are only finitely many such components, the conclusion follows.
\end{proof}

Finally, let us get rid of Assumption A:

\begin{lemma}
For an arbitrary semi-local component $U$,
any  component   $W $  of the horizontal slice $\bL_0 \cap U$, and
any integer $n \geq 0$,
the orders of tangency of the holomorphic disk $f^n ( W )$ with the
dynamical vertical foliation $\FF_{U^n  }$ are bounded
(where $U^n$ is the semi-local component containing $f^n(U)$).
\end{lemma}

\begin{proof}
We already know this for components satisfying Assumption A,
so let us deal with other components.

Assume $U^n\subset \NN(J^+)$, $n=0,1,\dots$.
Then the vertical dynamical foliations on the $U^n$
are tangent to the vertical line field $E^v$.
On the other hand,
by the choice of $\bL_0$, the slice $W$ is transverse to this line field.
Thus, $W$ is transverse to $\FF_U$.
By invariance of the dynamical foliation, the forward iterates $f^n(W)$ are
transverse to $\FF_{U^n}$: no tangencies in play.

This leaves us with finitely many components $U$.
For each of them, $W$ has finitely many tangencies with $\FF_U$
counted with multiplicities (by construction of $\bL$).
By invariance of the dynamical foliations,
$f^n(W)$ has the same number of tangencies with $\FF_{U^n}$
for any integer $n\ge0$. The conclusion follows.
\end{proof}

\begin{defn}\label{degcrit}[$\deg_\crit$]
Let  $\deg_\crit$ be the maximum of the orders of tangency
   that appear in the above lemma.
\end{defn}

\begin{corollary}\label{bound on the order of tan}
For any semi-local component $U^0$, any preimage $U^{-j}$,
and any  component   $W^{-j}$  of the horizontal slice $\bL_0 \cap U^{-j}$,
the order of tangency of the holomorphic disk $f^j (W^{-j})$ with the
dynamical vertical foliation $\FF_{U^0}$ is bounded by $\deg_\crit$.
\end{corollary}

\section{Final preparations}

The following is a rephrasing of Corollary \ref{cor:absorbing}.

\begin{lemma}\label{lemma:omega}
Let $\epsilon>0$.
Then there exists a domain 
$\Omega \subset \Delta^2_R$ for which
$$
f^{-1}(\Omega) \cap \Delta^2_R \subset \Omega,
$$
and
$$
\left(J^+ \cap \Delta^2_R \right)
\bigcup \mathrm {(semi-local\  wandering\ components) }
\setminus \{\; \mathrm{parabolic \; cycles}\; \} \subset \Omega,
$$
and which satisfies two conditions:
\begin{enumerate}
\item[ {\rm (i)} ] For any $z \in J^+$ there exist $w \in \Delta^2_R$
  with $|z-w| < \epsilon$ such that the semi-local leaf though $w$ does not intersect $\Omega$.
\item[{\rm  (ii) } ] Any $z\in \Omega$ that does not lie in a wandering Fatou component lies $\epsilon$-close to $J^+$.
\end{enumerate}
\end{lemma}

In several instances of the proof the value of the constant $\epsilon>0$ must be sufficiently small, which we will refer to by saying that $\Omega$ should be \emph{sufficiently thin}. The domain $\Omega$ will only be fixed after all bounds on degrees and diameters are determined. To avoid a circular argument we should take care that the constants that appear in those bounds can be defined independently of the exact choice of $\Omega$. At this time we only guarantee that $\Omega$ is chosen sufficiently thin so that every point
$z \in \Omega$ lies on a semi-local leaf of the artificial vertical lamination.

\comm{Let $\hat\Omega$ be the saturation of $\Omega $ by these leaves.
Let $\Omega_0$ be the slice of $\hat \Omega$  by the horizontal line $\mathbb L_0$.
We will also refer to $\Omega_0$ as the {\em holonomy projection} of
$\Omega$ to $\mathbb L_0$.
A possibly smaller set $\Omega_0^\prime$ is given by those points $z\in \Omega_0$ for which the entire semi-local vertical leaf through $z$ is contained in $\Omega$.

\begin{defn}[\emph{transverse diameter}]
Let $V\subset \Omega$ be an embedded holomorphic disk.
Let $V_0$ be the holonomy projection of $V$ to $\Omega_0$.
The \emph{transverse diameter} of $V$, denoted by
$
\mathrm{diam}_{\Omega_0} V,
$
is the maximal hyperbolic diameter  in $\Omega_0$ \marginpar{added}
 of a connected component of $V_0$.

The holonomy maps from $V$ to the straight disk $\mathbb L_0$ are quasiregular and have uniformly bounded dilatations. Hence for sufficiently small disks $V$ there exists a constant $\mu^\prime>0$ such that
$$
\mathrm{diam}_{\Omega_0} V  \le \mu^\prime \cdot \mathrm{diam}_{\Omega} V,
$$
where the latter refers to the Kobayashi diameter in $\Omega$.

At one point in the proof we will refer to the \emph{Euclidean transverse diameter} of $V$, by which we mean the maximal Euclidean diameter of the connected components of $V^0$.
\end{defn}
}

\begin{lemma}\label{lemma:basic2D}
Given constants $d \in \mathbb N$, $r<1$ and $\mu>0$ there exists a constant $C = C(r, d,\mu)$ with the following property. For any hyperbolic Riemann surface $V$, and any proper quasiregular map $f: V \rightarrow \mathbb D$ of degree at most $d$ whose dilatation is bounded by $\mu$, the hyperbolic diameter of any connected component of $f^{-1} D_r(0)$ is bounded by $C$. For fixed $d, \mu$ the constant $C(r, d, \mu)$ converges to $0$ as $r\rightarrow 0$.
\end{lemma}
\begin{proof}
The quasiregular map can be written as the composition of a proper quasiconformal homeomorphism with dilation bounded by $\mu$, with a proper holomorphic map of degree at most $d$. The statement holds for both of these maps, and hence also for the composition.
\end{proof}

\begin{defn}[\emph{protected lifts}]
Let $S \subset \Delta^2_R$ be a properly embedded holomorphic disk,
bounded away from the horizontal boundary $\{|y| = R\}$.
For a disk $D \subset \mathbb L_0$ consider all artificial vertical leaves through points in $D$. Write $V \subset S$ for a connected component of the set of intersection points of these artificial vertical leaves with $S$.

We say that $V$ is a \emph{lift} of $D$ if the holonomy correspondence
is \emph{proper},
 i.e. for any compact $E \subset D$ the intersection points of the
 leaves through $E$ with $V$
is compact.

We will consider lifts in $f^j \mathbb L_0 \cap \Omega$. Let $t > 1$
and consider two concentric disks $D_r(z) \subset D_{t\cdot r}(z)
\subset \mathbb L_0$. If the disks $D_r(z)$ and $D_{t\cdot r}(z)$ can
be lifted to $V_r(z) \subset V_{t \cdot r}(z) \subset f^j \mathbb L_0
\cap \Omega$, then we say that $V_r(z)$ is a \emph{protected lift} of $D_r(z)$.
We define the \emph{degree} of $V_r(z)$ as the maximal number of intersections with artificial vertical leaves.
\end{defn}

Recall that the artificial vertical leaves intersect $\mathbb L_0$ transversally near $J^+$.
 Thus, for sufficiently small disks $D_r(z) \subset \mathbb L_0$
 sufficiently close to $J^+$ the lift is traditional: a pullback under
 the holonomy map.
 However, if $D_r(z)$ intersects an artificial leaf non-transversally
then the holonomy from $V_r(z)$ to $D_r(z)$ cannot be single valued,
so we talk about the holonomy correspondence.

The properness of lifts is not automatic, and may be violated when an artificial vertical leaf through a disk $D$ is tangent to the boundary of $\Delta^2_R$. It is therefore possible that a disk $D_r(z)\subset \mathbb L_0$ that cannot be lifted, even when all artificial vertical leaves through $D$ are contained in $\Omega$. However, for every point $z \in \mathbb L_0$ for which the vertical leaf through $z$ intersects $f^j \mathbb L_0 \cap \Omega$ there is a sufficiently small disk that can be lifted. Conversely every point in $f^j \mathbb L_0 \cap \Omega$ is contained in some lift $V_r(z)$.

\begin{lemma}\label{constantC1}
Let $V_r(z) \subset V_{t \cdot r}(z)$ be a protected lift of degree $d$. Then there exist a constant $C_1(\frac{1}{t}, d)>0$ such that
$$
\mathrm{diam}_{\Omega} V_r(z) \le C_1(\frac{1}{t}, d),
$$
and given $d$ the constant $C_1(\frac{1}{t},d)$ converges to $0$ as $t \rightarrow \infty$.
\end{lemma}
\begin{proof}
If $D_{t \cdot r}$ intersects artificial vertical leaves at most once, then the holonomy defines a proper quasi-regular map from $V_{t \cdot r}$ to $D_{t\cdot r}$. The degree of the holonomy map is then exactly the maximal number of intersections of $V_{t\cdot r}$ with leaves of the artificial lamination, which is $d$. It follows from Lemma \ref{lemma:basic2D} that the Poincar\'e diameter
$$
\mathrm{diam}_{V_{t\cdot r}(z)} V_r(z)
$$
is then bounded by $C(\frac{1}{t}, d, \mu)$, for a bound $\mu$ on the order of qc-dilatation of the holonomy maps induced by the artificial vertical lamination. By definition of protected lifts we have $V_{t\cdot r}(z) \subset \Omega$, which implies the same bound on the Kobayashi diameter in $\Omega$.

Recall that the vertical lamination of $J^+$ is transverse to $\mathbb L_0$, hence the above discussion applies to sufficiently small disks in a sufficiently small neighborhood of $J^+$. By choosing $\Omega$ sufficiently thin, it follows that $D_{t\cdot r}(z) \subset \mathbb L_0$ intersects each vertical leaf in a unique point, unless $D_{t\cdot r}(z)$ intersects one of finitely many wandering Fatou components.

It is clear that for each given $D_r(z) \subset D_{t\cdot r}(z) \subset \mathbb L_0$ there does exist a bound on
$$
\mathrm{diam}_{V_{t\cdot r}(z)} V_r(z),
$$
depending only on the degree of $V_{t\cdot r}(z)$. Hence by compactness we obtain a bound when the radius $r$ is bounded away from zero. But when $r$ is sufficiently small, the disk $D_{t\cdot r}(z)$ either lies in a neighborhood of $J^+$ that guarantees that $D_{t\cdot r}(z)$ intersects each artificial vertical leaf at most once, or $D_{t\cdot r}(z)$ lies well inside one of a semi-local wandering Fatou component $U$, where $U$ is one of at most finitely many such components. It follows that
$$
\mathrm{diam}_{U} V_r(z) \rightarrow 0
$$
as $r \rightarrow 0$. Since $U \subset \Omega$, this completes the proof.
\end{proof}

\begin{lemma}\label{lemma:transverse}
There exists a constant $\epsilon>0$ such that the following holds. Let $V \subset \mathbb L_0$ be a holomorphic disk, and write $V^j = f^j(V)$. Suppose that for $j = 0, \ldots , n$ we have $V^j \in \Delta^2_R$ and
$$
\sup_{z \in V^j} d(z, J^+) < \epsilon,
$$
where $d(\cdot, \cdot)$ refers to the Euclidean distance in $\mathbb C^2$.
Then each $V^j$ is transverse to the artificial lamination. If the Euclidean diameter of each $V^j$ is sufficiently small then it follows moreover that each $V^j$ has degree $1$.
\end{lemma}
\begin{proof}
Recall that $\mathbb L_0$ is transverse to the artificial vertical lamination in a small neighborhood of $J^+$. Since the disks $V^j$ remain in the region of dominated splitting, their tangent spaces lie in some large horizontal cone field, while in a small neighborhood of $J^+$ the tangent spaces to the vertical leaves do not intersect those horizontal cones. Thus transversality follows, in fact with uniform bounds on the angles between the tangent spaces of the disks $V^j$ and the leaves of the artificial vertical lamination. It follows that sufficiently small disks will have degree $1$.
\end{proof}

\begin{lemma}\label{lemma:2D-N0}
Let $t > 1$. There exists an $N_0 = N_0(t) \in \mathbb N$ such that the following holds. For each protected lift $V_r(z) \subset V_{t\cdot r}(z)$ of disks $D_r(z) \subset D_{t\cdot r}(z) \in \mathbb L_0$, the preimage $V_r^{-1}(z)$ can be covered by protected lifts of at most $N_0$ disks $D_{r_k}(z_k) \subset D_{2t\cdot r_k}(z_k)$.

Moreover,
in the particular case where the disk $D_{t \cdot r}(z)$ lies in a
semi-local wandering component $U$,
the lifts $V_{2t\cdot r_k}(z_k)$ are all contained in the component $U^{-1}$. In all other cases, the lifts $V_{2t\cdot r_k}(z_k)$ are all  contained in $V_{tr}^{-1}(z)$.
\end{lemma}
\begin{proof}
It is clear that for every $r>0$ and $z \in \mathbb L_0$ and every choice of protected lift $V_r(z) \subset V_{t \cdot r}(z)$, the disk $V_r^{-1}(z)$ can be covered with a finite number of protected lifts $V_{r_k}(z_k)$ for which $V_{2t r_k}(z_k) \subset V_{tr}^{-1}(z)$. Suppose for the purpose of a contradiction that there exists a sequence $(r, z, V_r(z))$ for which the minimal number of lifts needed converges to infinity.

By restricting to a subsequence we may assume that the lifts $V_{t\cdot r}(z)$ are either contained in wandering Fatou components, in periodic Fatou components, or all intersect $J^+$. We consider theses cases separately.

First suppose that the lifts $V_{t\cdot r}(z)$ are contained in wandering Fatou components. Suppose first that $V_{t\cdot r}(z)$ is contained a wandering domain $U$ where the artificial lamination is backwards invariant. By invariance the image of $V^{-1}_{t\cdot r}(z)$ in $\mathbb L_0$ under holonomy is independent of the choice of lift, and the argument is the same as in the one-dimensional setting: For disks of radius bounded away from zero the bound on $N_0$ follows from compactness. But sufficiently small disks are either contained in a small neighborhood of $J^+$, where the lamination is transverse to $\mathbb L_0$, or well inside wandering domains. In the latter case it is clear that $V^{-1}_r(z)$ can in fact be covered by a single lift. In the former case holonomy induces a quasiconformal map of bounded dilatation, which gives a bound on the distortion and thus on $N_0$.

Since the vertical lamination is invariant except in finitely many components, we may therefore assume that all $V_{t\cdot r}(z)$ are in one of the wandering components for which the vertical lamination is not backwards invariant. Recall that in this wandering Fatou component the vertical lamination is still invariant in a neighborhood of the boundary $J^+$. The bound on $N_0$ follows again when $r$ remains bounded away from zero, hence we may assume that $r \rightarrow 0$. But in that case the lifts are either very close to $J^+$, where the lamination is invariant and the bound on $N_0$ follows as above, or the lifts are bounded away from $J^+$. But then for sufficiently small $r$ the preimage $V_r^{-1}(z)$ can again be covered by a single lift $V_{\rho}(w)$ for which $V_{2t\cdot \rho}(w) \subset U^{-1}$.

Now suppose that the lifts $V_{t\cdot r}(z)$ are contained in periodic Fatou components. Recall that $\mathbb{L}_0$ was chosen to be transverse to the dynamical lamination near $J^+$. Thus, by making $\Omega$ sufficiently thin, we may assume that both $V_{t\cdot r}(z)$ and $V_{t\cdot r}^{-1}(z)$ are transverse to the vertical lamination, with angles bounded from below, see Lemma \ref{lemma:transverse}. Again it follows that holonomy from $\mathbb L_0$ to $V_{t\cdot r}^{-1}(z)$ and from $V_{t\cdot r}(z)$ back to $\mathbb L_0$ induces a quasiconformal map of bounded dilatation, which implies a bound on $N_0$.

The last case to be considered is when the lift $V_{t\cdot r}(z)$ intersects $J^+$. Suppose first that there exists a subsequence for which $r$ converges to zero. In that case $n$ must converge to infinity, since otherwise the preimages $V^{-i}_{t\cdot r}(z)$ are all contained in the domain of dominated splitting, which implies that the lifts $V_{t\cdot r}$ are horizontal, giving a bound on $N_0$. In fact, by Lemma \ref{newlemma4} the contraction in the horizontal direction is sub-exponential, and hence in backwards time the expansion is sub-exponential, therefore we may assume that
$$
n \ge \log_\alpha(r)
$$
for any $\alpha>1$ and $r$ sufficiently small.

Since $n$ is large the disks $V_{t\cdot r}(z)$ must intersect $J^+$ in a point where $G^-$ is very close to zero, which implies that there is a nearby point $x \in J$ on the same local stable manifold. For each $x\in J$ there exists a small closed loop $\gamma$ around $x$ in $W^s(p)$ where $G^-$ is strictly positive. Consider a small disks through each point in $\gamma$, normal to $W^s(p)$, and therefore in particular horizontal, and denote the union of these horizontal disks by $\Gamma$. By making these horizontal disks sufficiently small, we can guarantee that $G^-$ is strictly positive on $\Gamma$. It follows that there exists an $N$ such that
$$
f^{-N}(\Gamma) \cap \Delta^2_R = \emptyset.
$$
By compactness of $J$ , we can find  a uniform bound from above on the diameter of the $\gamma$, a uniform bound from below on the size of the horizontal disks through $\gamma$, and a uniform bound on $N$.

By the exponential contraction in the vertical direction and the fact that the inverse images $f^{-j}(V_{tr}(z))$ are all contained in the bidisk $\Delta^2_R$, it follows that $V_{tr}(z)$ is contained in a tubular neighborhood with radius of order $r^\alpha$ around the forward image of a horizontal disk $D$, (with radius of order $t\cdot r$) in a $f^{n-N} \Gamma(p)$, for some point $p \in J$, with similar estimates for $V^{-1}_{t\cdot r}(z)$ and $f^{-1}(D)$. Therefore it is sufficient to consider the respective lifts in $D$ and $f^{-1}(D)$, and the fact that these disks are horizontal implies a bound on $N_0$ by the same argument as above.

Thus the remaining lifts $V_{t\cdot r}(z)$ we need to consider have radius bounded away from zero and intersect $J^+$. By making $\Omega$ sufficiently thin we can therefore guarantee that for some $1< t_2 < t$ the lift $V_{t_2 \cdot r}(z)$ is contained in a wandering domain $U$. The existence of the bound $N_0$ follows by the same argument as when $V_{t\cdot r}(z) \subset U$.


\end{proof}

\begin{figure}[t]
\centering
\includegraphics[width=2.5in]{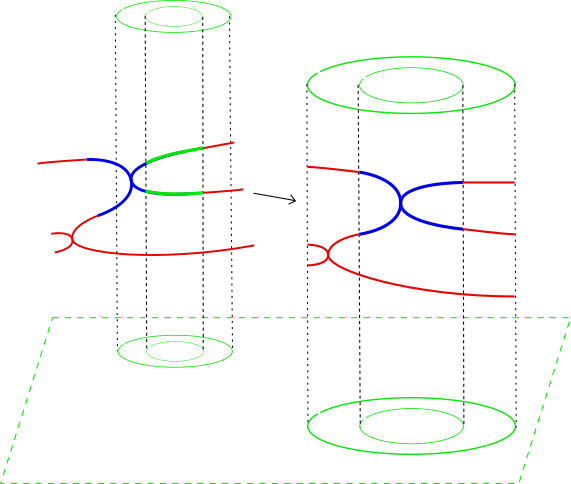}
\caption{The set $V_{r}^{-1} $ is covered by lifts of disks $D_{r_k}(z_k) \subset D_{t\cdot r_k}(z_k)$.}
\label{figure:outline}
\end{figure}

Figure \ref{figure:outline} illustrates the covering of $V_r^{-1}(z)$ by lifts of disks $D_{r_k}(z_k)\subset D_{t\cdot r_k}(z_k)$. The lift $V_r(z)$ and its inverse image $V_r^{-1}(z)$ are depicted in blue; the rest of the larger $V_{tr}(z)$ and $V_{t\cdot r}^{-1}(z)$ in red. In the sketch the vertical lamination is given by straight vertical lines. Two distinct lifts of a single disk $V_{r_k}(z_k)$ are depicted in green.

\begin{remark}
By construction the bound $N_0$ is independent of $\Omega$, that is, $N_0$ does not need to be changed when $\Omega$ is made smaller.
\end{remark}

Let $U$ be a semi-local wandering domain on which the artificial lamination is not equal to the dynamical lamination. We will compare lifts of disks $D_r(z) \subset \mathbb L\cap U$ with respect to both laminations.

\begin{prop}\label{2D-prop11}
There exists $\delta = \delta(U) >0$ such that the following holds. Let $D \subset \mathbb L_0 \cap U$ be a holomorphic disk of hyperbolic diameter (in $\mathbb L_0 \cap U$) at most $\delta$, and let $V \subset f^j \mathbb L_0$ be a lift with respect to the dynamical lamination. Then $V$ is contained in a lift with respect to the artificial lamination of a protected disk $D_\rho(w) \subset D_{2 \cdot \rho}(w) \subset \mathbb L_0 \cap U$.
\end{prop}
\begin{proof}
Note that any connected component of $\mathbb L_0 \cap U$ is simply connected. Therefore, as in the proof of Lemma \ref{1D-lemma4}, a subset $E \subset \mathbb L_0 \cap U$ of sufficiently small hyperbolic diameter is contained in a disk $D_r(w)$ satisfying $D_{r}(w)\subset D_{2\cdot r}(w) \subset U$. Hence the proof is completed by showing that $V$ is a lift (with respect to the artificial lamination) of a set $E\subset \mathbb L_0 \cap U$ of sufficiently small hyperbolic diameter.

Since the dynamical and artificial laminations coincide near the boundary of $U$, the statement holds trivially when $V$ is sufficiently close to the boundary. We will therefore consider lifts of disks $D$ that are bounded away from $\partial U$.

Recall from the previous section that the holomorphic disks $f^n(W^{-n})$ converge to the horizontal lamination. Both the horizontal leaves and the disks $f^n(W^{-n})$ cannot coincide with vertical dynamical leaves, hence for a sufficiently small tubular neighborhood $\mathcal{N}$ of a dynamical leaf, the intersection $f^n(W^{-n}) \cap \mathcal{N}$ have arbitrarily small Euclidean diameters. Since we consider leaves that are bounded away from the boundary, small Euclidean diameters imply small Kobayashi diameters in $U$, and thus $V$ can be assumed to be a lift of some $E \subset \mathbb L_0 \cap U$ of arbitrarily small hyperbolic diameter, which completes the proof.

\comm{
These disks cannot coincide with leaves of the dynamical vertical lamination. Thus, for sufficiently large $n$ and a sufficiently small tubular neighborhood $\mathcal{N}$ of a dynamical leaf, the connected components of $f^n(W^{-n}) \cap \mathcal{N}$ will have arbitrarily small Euclidean diameter, and thus have arbitrarily small transverse diameter with respect to the artificial vertical lamination if the dynamical leaf is bounded away from the boundary of $U^0$.

Thus, we have proved the statement for $V \subset f^n \mathbb L_0$ and $n\ge N$ for some $N \in \mathbb N$, leaving us with finitely many disks $f^n(W^{-n})$. Since those disks also cannot coincide with leaves of the dynamical vertical lamination, it again follows that intersections with sufficiently small tubular neighborhoods $\mathcal{N}$ will have arbitrarily small Euclidean diameter. This completes the proof.
}
\end{proof}

\comm{
For each semi-local wandering domain we have the dynamical vertical lamination, obtained by pulling back the dynamical vertical lamination of a component $U^n$, for $n$ sufficiently large. We also have a family of horizontal leaves given by the disks $f^n(W^{-n})$. There is a maximal order of tangency between these two families of Riemann surfaces, which we can regard as the local degree on the component $U^0$, a notion that corresponds to the degree caused by critical points in the one-dimensional setting. We let $\mathrm{deg}_{\mathrm{crit}}$ be the maximum of these upper bounds over all the possible semi-local wandering components. The maximum exists since the order of tangency is invariant under the automorphism $f$, and there are only finitely many grand orbits of semi-local wandering components for which this order of tangency is not equal to $1$.}

\begin{defn}\label{defn:beta}[$\beta$]
Let $\Lambda_R(a)$ be a vertical leaf that is either contained in $J^+$ or in a semi-local wandering domain, and assume that $\Lambda_R(a)$ is not one of the finitely many parabolic leaves in $J^+$ that were removed from $\Omega$. Then the leaf $\Lambda_R(a)$ is contained in $\Omega$. Note that the sets $f^j \mathbb L \cap \Delta^2_R$ stay bounded away from the horizontal boundary $\{|w| = R\}$. It follows that there exists an upper bound $\beta$, independent from $j \in \mathbb N$, on the Kobayashi diameter in $\Omega$ of the intersection of any vertical leaf $\Lambda_R(a)$ with any $f^j \mathbb L_0$. We note in particular that the constant $\beta$ can be chosen independently of $\Omega$.
\end{defn}

\begin{defn}\label{diammax}[\emph{$\mathrm{diam}_{\mathrm{max}}$}]
For $t>1$ and an integer $d \ge 2$ we define
$$
\mathrm{diam}_{\mathrm{max}}(t,d) := 2 N_0(t) \cdot C_1(\frac{1}{2}, d) + \beta.
$$
In what follows $t$ will equal either $2$ or $2K$, where the constant $K$ will be introduced in Proposition \ref{2D-prop12}. The integer $d$ will equal either $1$ or $\deg_\crit$.
\end{defn}

\begin{defn}\label{bigcomponent}[\emph{big wandering domain}]
We say that a semi-local wandering domain $U$ is a \emph{big wandering
  domain} if the vertical lamination on $U$ is not backwards invariant,
 or if there exist $j \in \mathbb N$ and $S \subset U \cap f^j \mathbb L_0$
 of Kobayashi diameter
$$
\diam_U S \le \diam_{\max} (2,1)
$$
intersecting a vertical leaf $\Lambda_R(a)$ for which $f^{-1}\Lambda_R(a) \cap \Delta^2_R$ has more than one component intersecting $f^{-1} S$.
\end{defn}

We note that there are at most finitely many big wandering domains.

\begin{defn}\label{regularcomponent}[Regular wandering domains]
A semi-local wandering domain $U$ is said to be \emph{regular} if none of the components $U^{-n}$ for $n \ge 0$ are big wandering domains.

There exists only finitely many bi-infinite orbits of semi-local wandering domains that are not regular. A component that is not regular is called \emph{post-critical}. Note that there at most finitely many grand orbits of semi-local wandering Fatou components that contain post-critical domains. If $U^n$ is regular but $U^{n+1}$ is post-critical then we say that $U^n$ is \emph{critical}.
\end{defn}

\begin{defn}\label{degmax}[\emph{$\mathrm{deg}_{\mathrm{max}}$}]
Recall from Lemma \ref{lemma:forwarddegree} that the degree of the maps $f^n: U^0 \rightarrow U^n$ is bounded from above by a constant independent of $n$. Thus, such a forward degree bound exists for each critical wandering component. Since there are only finitely many critical components, there exists a uniform bound, which we will denote by $\mathrm{deg}_{\mathrm{max}}$, analogously to the one-dimensional setting.
\end{defn}

\section{Diameter and degree bounds - proof for H\'enon maps}

Let us recall the constants and objects that play a role in the upcoming proofs, listed in the order of their dependency.

\begin{enumerate}
\item[]$\mathrm{deg}_{\mathrm{crit}}$ : The maximal local degree (Def \ref{degcrit}).
\item[]$\mathbb L_0$ : Convenient choice of horizontal line (Def. \ref{ellnot}).
\item[]$\beta$ : Upper bound on the Kobayashi diameter in artificial vertical leaves (Def. \ref{defn:beta}).
\item[]$N_0(t)$ : Maximal number of $t$-protected disks whose lifts cover $V_r^{-1}(z)$ (Lemma \ref{lemma:2D-N0}).
\item[]$C_1(\frac{1}{t},d)$ : Upper bound on the Kobayashi diameter for lifts in $\Omega$ (Lemma \ref{constantC1}).
\item[]$\mathrm{deg}_{\mathrm{max}}$ : Bound on global degrees of iterates on critical wandering domains (Def. \ref{degmax}).
\item[]$K$: Defined in Prop. \ref{2D-prop12} below. We will consider protected lifts $V_r(z) \subset V_{2K\cdot r}(z)$.
\item[]$\mathrm{diam}_{\mathrm{max}}(t,d) := 2 N_0(t) \cdot C_1(\frac{1}{2}, d) + \beta$ (Def. \ref{diammax}).
\item[]$\Omega$ : Domain of consideration, chosen sufficiently thin
  (Lemma \ref{lemma:omega}).
\end{enumerate}

In this section we prove the main estimates on the diameters and
degree of lifts $V_r(z)$ and their preimages. Just as in the
one-dimensional argument we distinguish between three different kinds of
lifts. First we consider lifts that are deeply contained in wandering
components, i.e. lifts of disks $D_r(z)$ for which $D_{K\cdot r}(z)$
is contained in the same component for a sufficiently large
constant $K$. Afterwards we consider the two remaining cases, namely
lifts of disks that are not contained in wandering components,
and lifts of disks $D_r(z)$ that are contained in wandering components
but for which the
protecting disks $D_{K \cdot r}(z)$ are not.

\subsection{Lifts deeply contained in wandering domains.}

In what follows we let $(U^n)_{n \in \mathbb Z}$ be a bi-infinite
orbit of semi-local wandering Fatou components, and consider protected lifts $V_r \subset V_{t\cdot r}$ for which $V_{t\cdot r}$ is contained in one of the domains $U^n$. As  post-critical components are harder to deal with than regular components, we start with the latter.

\begin{lemma}\label{2D-lemma12}
Let $U^n$ be a regular wandering semi-local Fatou component and let $V_r(z) \subset V_{2r}(z) \subset f^j\mathbb L_0 \cap U^n$ be protected lifts of disks $D_r(z) \subset D_{2r}(z) \subset \mathbb L_0$. Then for $i = 0, \ldots ,j$ we have
$$
\mathrm{diam}_{\Omega} V_r^{-i}(z) \le \diam_{\max} (2,1)
$$
and
$$
\mathrm{deg} V_r^{-i}(z) = 1.
$$
\end{lemma}

\begin{proof}
We will assume the statement holds for a given $j \in \mathbb N$, and
proceed to prove it for $j+1$. By Lemma \ref{lemma:2D-N0} the
holomorphic disk $V^{-1}_r(z)$ can be covered by lifts $V_{r_k}(z_k)$
of at most $N_0$  disks $D_{r_k}(z_k) \subset D_{4\cdot r_k}(z_k) \subset \mathbb L_0 \cap U^{n-1}$. By the induction assumption each $V_{2\cdot r_k}^{-i}(z_k)$ has degree $1$, and thus we obtain the estimates
$$
\mathrm{diam}_{U^{n-i-1}} V_{r_k}^{-i}(z_k) \le C_1(\frac{1}{2},1 ).
$$
Note that a priori we do not have a bound on the number of lifts $V_{r_k}(z_k)$.
Since $U^n$
is regular, the vertical lamination on the components $(U^{n-i})_{i \ge 0}$ is invariant. Let $x \in V_r(z)$, and write $x_{-i} =
f^{-i}(x)$. Each point in a lift $V_{r_k}(z_k)$ can be connected to the semi-local leaf $\Lambda(x_{-1})$ by a path that travels through at most $N_0$ other lifts $V_{r_l}(z_l)$. Hence it follows that
$$
\mathrm{diam}_{U^{n-1}} V_{r}^{-1}(z) \le \diam_{\max} (2,1).
$$
Recall that by Definitions \ref{bigcomponent} and \ref{regularcomponent} of respectively \emph{big} and \emph{regular} Fatou components, this bound on the hyperbolic diameter in $U^{n-1}$ of $V_r^{-1}(z)$ implies that $\Lambda(x_{-2})$ is the unique semi-local leaf in $f^{-1} \Lambda(x_{-1})  \cap \Delta^2_R$ that intersects $V_r^{-2}(z)$. The same argument as above gives the same bound on the hyperbolic diameter in $U^{n-2}$ of $V_r^{-2}(z)$. Continuing by induction on $i$ gives
$$
\mathrm{diam}_{U^{n-i}} V_{r}^{-i}(z) \le \diam_{\max} (2,1),
$$
for all $i \le j$. We obtain the required bounds on the hyperbolic diameters in $\Omega$, as $\Omega$ contains all semi-local wandering components. Since the semi-wandering component $U$ is regular it follows that
$$
\mathrm{deg} V_r^{-i}(z) = 1
$$
for $i = 0, \ldots ,j$.
\end{proof}

\begin{prop}\label{2D-prop12}
There exists a constant $K >0$ such that the following holds. Let $U^n$ be a semi-local wandering component, and let $V_r(z)\subset V_{K\cdot r}(z) \subset f^j\mathbb L_0 \cap U^n$ be protected lifts of disks $D_r(z) \subset D_{K\cdot r}(z) \subset \mathbb L_0$. Then for all $i \le j$ we have
$$
\mathrm{diam}_{\Omega} V_r^{-i}(z) \le \diam_{\max} (2,1),
$$
and
$$
\mathrm{deg} V_r^{-i}(z) \le \mathrm{deg}_{\mathrm{crit}}.
$$
\end{prop}
\begin{proof}
Write $U^{n-i}$ for the semi-local wandering Fatou component in $f^{-i} U^n \cap \Delta^2_R$ that contains the lift $V_r^{-i}(z)$. By renumbering the orbit $U^{j}$ we may assume that $U^0$ is critical. It follows from the previous lemma that $K \ge 2$ suffices when $n \le 0$, so let us suppose that $n > 0$.

We first consider the hyperbolic diameters in components $U^{n-i}$, for $i \le n$. We write $D_r(z) \subset D_{k\cdot r}(z) \subset U^n\cap \mathbb L_0$ for the disks giving the lifts $V_r(z) \subset V_{K\cdot r}(z)$. We also consider the lifts these disks to $f^{i} \mathbb L_0$.

Recall from the proof of Lemma \ref{existence} that for $n$ sufficiently large any connected component $D$ of $\mathbb{L}_0\cap U^n$ is a global transversal. The composition of $f^i$ with holonomy from $f^i(\mathbb L_0 \cap U^{n-i}) \rightarrow D$, induces quasi-regular maps, whose degrees are bounded by $\mathrm{deg}_{\mathrm max}$ and have bounded dilatation. It follows that by making $K$ sufficiently large, the hyperbolic diameter of the preimage in $\mathbb L_0 \cap U^{n-i}$ can be made arbitrarily small. Note that $V^{-i}_r(z)$ is a lift of this pre-image with respect to the \emph{dynamical lamination}. It follows that the Kobayashi diameter of $V^{-i}_r(z)$ in $U^{-i}$, and thus also in $\Omega$, can therefore be made arbitrarily small by choosing $K$ large.

Let us consider the finitely many values of $n$ for which we do not know that $D$ is a global transversal. For any fixed $n$, making $K$ sufficiently large implies that $V_r^{-i}(z)$ is contained in an arbitrarily small neighborhood of a dynamical leaf, and thus we immediately obtain the same bounds on the hyperbolic diameter in $\mathbb L_0 \cap U^{n-i}$. Therefore we can drop the assumption that $n$ is sufficiently large.

We now consider the artificial vertical lamination on $U^0$. We have seen that $V_r^{-i}(z)$ is the lift with respect to the dynamical lamination of a set $E\subset \mathbb L_0 \cap U^0$, whose hyperbolic diameter can be made arbitrarily small by choosing $K$ sufficiently large.  Lemma \ref{2D-prop11} implies that $V_r^{-i}(z)$ is contained in the lift of a disk $D_{\rho}(w) \subset D_{2 \cdot \rho}(w) \subset \mathbb L_0 \cap U^0$ with respect to the artificial lamination when the hyperbolic diameter of $E$ is less than $\delta = \delta(U^0)$. Since there are only finitely many critical components $U^0$, the constant $\delta$ can be chosen independently of the critical component, and thus $K$ can be chosen independently as well.

By applying Lemma \ref{2D-lemma12} to the lift of the disk $D_\rho(w)$ we obtain the required diameter bounds on $V^{-i}_r(z)$ for $n \le i \le j$, as well as
$$
\mathrm{deg} V^{-i}_r(z) = 1
$$
for $n \le i \le n$. Hence if $K$ is chosen sufficiently large it follows that
$$
\mathrm{deg} V^{-i}_r(z) \le \mathrm{deg}_{\mathrm{crit}}
$$
for all $i \le j$.
\end{proof}

\subsection{Lifts that are not deeply contained.}

We are ready to prove the main technical result:

\begin{prop}\label{Prop51}
Let $z \in \Omega$ and let $r>0$ be such that $D_{2K\cdot r}(z) \subset \Omega$. Then for every $j \in \mathbb N$ and any lift $V_r(z) \subset f^n \mathbb L_0$ one has:
$$
\mathrm{diam}_{\Omega} V_r^{-j}(z) \le \mathrm{diam}_{\mathrm{max}}(2K, \deg_\crit ),
$$
and
$$
\mathrm{deg} V_r(z) \le \mathrm{deg}_{\mathrm{crit}}.
$$
\end{prop}
\begin{proof}
Having previously dealt with lifts that are sufficiently deeply contained in wandering domains, we only need to consider the two other cases: lifts are either not contained in a wandering component, or those that are contained in a wandering component, but lifts of disks $D_r(z)$ for which $D_{K \cdot r}(z)$ intersects $J^+$.

We again assume the induction hypothesis that both bounds hold for all $i \le j$, and proceed to prove the statements for $j+1$. We will first prove the induction step under the assumption that $D_r(z)$ is not contained in a wandering domain, and afterwards deal with the case where $D_r(z)$ is contained in a wandering domain. The result in the former case will be used to prove the latter.

\smallskip

\emph{Case 0.} Let us first consider the situation where $V_r(z)$ is contained in a periodic Fatou component. By choosing $\Omega$ sufficiently thin we can guarantee that $V_r(z)$ is horizontal and that the disks $D_r(z)$ and $D_{2k\cdot r}(z)$ are arbitrarily small. It follows that $V_r(z)$ must have degree $1$. Hence we can cover $V_r^{-1}(z)$ with at most $N_0(2K)$ protected lifts $V_{r_k}(z_k) \subset V_{4K \cdot r_k}(z_k) \subset V_{2K \cdot r}(z)$. It follows by the induction assumption that the disks $V_{2\cdot r_k}(z_k)$ all have degree at most $\mathrm{deg}_{\mathrm{crit}}$. Recall from the proof of Lemma \ref{constantC1} that the Poincar\'e diameter of
$V_{r_k}(z_k)_\nu$ with respect to $V_{2\cdot r_k}(z_k)_\nu$ is
bounded by $C_1(\frac{1}{2}, \mathrm{deg}_{\mathrm{crit}})$. Since
$f$ is an automorphism,
we obtain the same bounds for the pairs $V^{-i}_{r_k}(z_k)_\nu \subset V^{-i}_{2\cdot r_k}(z_k)_\nu$, and since each $V^{-i}_{2\cdot r_k}(z_k)_\nu \subset \Omega$, therefore we have the bounds
$$
\mathrm{diam}_{\Omega} V^{-i}_{r_k}(z_k) \le C_1(\frac{1}{2}, \mathrm{deg}_{\mathrm{crit}}).
$$
This implies that
$$
\mathrm{diam}_{\Omega} V_r^{-i}(z) \le N_0(2K) \cdot C_1(\frac{1}{2}, \mathrm{deg}_{\mathrm{crit}}),
$$
which in turn gives the necessary degree bounds.

\smallskip

\emph{Case 1.} Now assume that $V_r(z)$ is not contained in a Fatou component. Then there exists $y_0 \in D_r(z) \cap J^+$, and we let $x_0 \in V_r(z)$ lie in the semi-local vertical leaf through $y_0$, which we denote by $\Lambda_R(x_0)$. We similarly write $\Lambda_R^{-i}(x_0)$ for the semi-local leaf through $f^{-i}(x_0)$, which by invariance of the vertical lamination on $J^+$ equals the connected component of $f^{-i}(\Lambda_R(x_0))$ that contains $f^{-i}(x_0)$.

Cover $V^{-1}_r(z)$ by protected lifts
$V_{r_k}(z_k)_{\nu} \subset V_{4K\cdot r_k}(z_k)_\nu \subset \Omega$ of at most $N_0(2K)$ disks $D_{r_k}(z_k) \subset D_{4K\cdot r_k}(z_k) \subset
\mathbb L_0$. The induction hypothesis implies that the degree of each lift
$V_{2\cdot r_k}(z_k)_{\nu}$ is bounded by
$\mathrm{deg}_{\mathrm{crit}}$, hence as above we obtain
$$
\mathrm{diam}_{\Omega} V^{-i}_{r_k}(z_k)_\nu \le C_1(\frac{1}{2}, \mathrm{deg}_{\mathrm{crit}}).
$$

As before we do not have an a priori estimate on the number of lifts $V_{r_k}(z_k)_\nu$. We apply the same induction argument as in Proposition \ref{2D-prop12}. By choosing $\Omega$ sufficiently thin we can guarantee that the Euclidean diameter of $V_r(z)$ is sufficiently small, so that for each vertical leaf $\Lambda_R(a)$ intersecting $V_r(z)$ there is a unique component of $f^{-1}\Lambda_R(a) \cap \Delta^2_R$ intersecting $V_r^{-1}(z)$.

Any point in $V^{-1}_r(z)$ can then be connected to a point in $\Lambda_R^{-1}(x_0)$ by a path that passes through at most $N_0(2K)$ lifts $V_{r_k}(z_k)_\nu$. By the definition of $\beta$ it follows that
$$
\mathrm{diam}_\Omega V_r^{-1}(z)  \le \mathrm{diam}_{\mathrm{max}}(2K, \deg_\crit ).
$$
Hence we can continue the induction procedure, and obtain
$$
\mathrm{diam}_\Omega V_r^{-i}(z) \le \mathrm{diam}_{\mathrm{max}}(2K, \deg_\crit )
$$
for all $i \le j$.

By choosing $\Omega$ sufficiently thin it follows from Lemma \ref{lemma:transverse} that each $V^{-i}_r(z)$ is transverse to the artificial vertical lamination and that
$$
\mathrm{deg} V^{-i}_r(z) = 1
$$
for all $i \le j$, completing the proof for disks $D_r(z)$ that are not contained in a wandering component.

\smallskip

\emph{Case 2.} In the remainder of this proof we assume that $D_r(z)$ is contained in a semi-local wandering component $U$, but that $D_{K\cdot r}(z)$ is not contained in $U$. We again cover $V^{-1}_r(z)$ with protected lifts $V_{r_k}(z_k)_\nu \subset V_{4K\cdot r}(z_k)_\nu \subset f^{j-1}\mathbb L_0 \cap \Omega$ of at most $N_0(2K)$ disks $D_{r_k}(z_k) \subset D_{4K \cdot r}(z_K) \subset \mathbb L_0$, and by induction obtain the diameter bounds
$$
\mathrm{diam}_\Omega V_{r_k}^{-i}(z_k)_\nu \le \mathrm{diam}_{\mathrm{max}}(2K, \deg_\crit ).
$$
Again the difficulty lies in the fact that a priori we do not have a bound on the number of lifts $V_{r_k}(z_k)_\nu$. As before we will obtain the diameter bounds for each inverse image $V_r^{-i}(z)$ by connecting the disks $V_{r_k}^{-i+1}(z_k)_\nu$ to a single semi-local vertical leaf.

Write $W$ for the component of $U \cap \mathbb L_0$ that contains $D_r(z)$, and let $w \subset \partial W$ be such that $|z - w|$ is minimal, and write $[z, w] \subset \mathbb L_0$ for the closed interval.

Consider the disk $D_{\frac{r}{2}}(w)$. Since $D_{K\cdot r}(w) \subset D_{2K\cdot r}(z)$ it follows that $D_{K\cdot r}(w) \subset \Omega$. Hence the disk $D_{\frac{r}{2}}(w)$ satisfies the conditions of the previously discussed case, and we obtain the estimates
$$
\mathrm{diam}_{\Omega} V^{-i}_{\frac{r}{2}}(w) \le \mathrm{diam}_{\mathrm{max}}(2K, \deg_\crit )
$$
for any protected lifts $V_{\frac{r}{2}}(w) \subset V_{K\cdot r}(w) \subset f^j \mathbb L_0 \cap \Omega$ of the disks $D_{\frac{r}{2}}(w) \subset D_{K \cdot r}(w)$.

The interval $[z, w]$ can be covered by the disk $D_{\frac{r}{2}}(w)$, plus a bounded number of disks
$$
D_{s_1}(w_1), \ldots,  D_{s_{N_1}}(w_{N_1}),
$$
each satisfying $D_{K\cdot s_\nu}(w_\nu) \subset U^n$. The maximal number of disks needed is bounded by a constant $N_1 \in \mathbb N$ that only depends on $K$, see Figure \ref{figure:disks}.

Write
$$
V_{s_\nu}(w_\nu)_\xi \subset V_{K \cdot s_\nu}(w_\nu)_\xi \subset f^j\mathbb L_0\cap \Omega
$$
for protected lifts of the disks $D_{s_\nu}(w_\nu)\subset D_{K\cdot s_\nu}(w_\nu)$. It follows that these lifts satisfy the conditions of Proposition \ref{2D-prop12}, and hence their preimages satisfy
$$
\mathrm{diam}_\Omega V^{-i}_{s_\nu}(w_\nu)_\xi \le \mathrm{diam}_{\mathrm{max}}(2K, \deg_\crit ).
$$
As in case 1 we obtain a bound on the hyperbolic distance of each point $f^{-i}(z)$ to $\partial U^{-i} \subset J^+$. By choosing $\Omega$ sufficiently thin, it follows that the points $f^{-i}(z)$ can all be assumed to lie arbitrarily close to $J^+$. In particular we may assume that the vertical leaf through each $f^{-i}(z)$ is dynamical, and these semi-local leaves $\Lambda_R^{-i}(z)$ are in particular invariant under $f$. Recall from Definition \ref{defn:beta} that
$$
\mathrm{diam}_\Omega f^{j-i} \mathbb L_0 \cap \Lambda_R^{-i}(z)
$$
is bounded by $\beta$. Hence we can use the same argument as in case (1), using $\Lambda_R^{-i}(z)$ instead of a semi-local leaf in $J^+$, to obtain the required diameter and degree bounds.
\end{proof}

\section{Consequences}

As before we will assume that the H\'enon map $f$ is substantially dissipative and admits a dominated splitting on $J$.

\begin{lemma}
There are no wandering Fatou components.
\end{lemma}
\begin{proof}
Suppose for the purpose of a contradiction that there does exist a
wandering Fatou component $U$. Without loss of generality we may
assume that $U$ intersects $\Delta^2_R$. Since $U \cap \Delta^2_R$ is
foliated by semi-local strong stable manifolds, the intersection $U
\cap \mathbb L_0$ is a non-empty relatively open set. Let $D$ be a
holomorphic disk relatively compactly contained in $U \cap \mathbb
L_0$.
By construction, we may assume that $\Omega$ is sufficiently thin and $D \subset U \cap \mathbb L_0$ is sufficiently large such that the hyperbolic diameter of $D$ in $\Omega^\prime$ is larger than $\mathrm{diam}_{\mathrm{max}}$.

Since $U$ lies in the Fatou set there exists a sequence $(n_j)$ for which $f^{n_j}$ converges uniformly on compact subsets of $U$ to a map $h: U \rightarrow J$. By the dominated splitting the image $h(U)$ is a point $p \in J$, and without loss of generality we may assume that $p$ is not contained in a parabolic cycle. Let $D_r(p) \subset D_{2r}(p) \subset \Omega$ be a transverse disk. Since $D$ is relatively compact in $U$, it follows that for sufficiently large $j$ the image $f^{n_j}(D)$ is contained in a lift  of $D_r(p)$ in $f^{n_j}(\mathbb L_0$. But then it follows from Proposition \ref{Prop51} that the hyperbolic diameter of $D$ in $\Omega^\prime$ is bounded by $\mathrm{diam}_{\mathrm{max}}$, leading to a contradiction.
\end{proof}

Lacking wandering domains the proof of Proposition \ref{Prop51} becomes considerably simpler. An immediate consequence is a better degree bound.

\begin{corollary}
The constant $\mathrm{deg}_{\mathrm{crit}}$ can be taken equal to $1$.
\end{corollary}
\begin{proof}
Since there are no wandering Fatou components, it follows from the proof of Proposition \ref{Prop51} that the degree of the disk $V_r$ is bounded by the degree of the disk $V_n$, which is a small straight disk contained in $\mathbb L_0$. Recall that $y_0$ was chosen so that there are no tangent intersections between $\mathbb L_0$ and the dynamical vertical lamination on $J^+$. Now that we know that there are no wandering Fatou components, the set $\Omega$ can be constructed as an arbitrarily thin neighborhood of $J^+_R$. Hence we may assume that there are no tangencies in $\mathbb L_0$ with the artificial vertical lamination in $\Omega$. Since the disks $V_n$ may be assumed to have arbitrarily small Euclidean diameter, it follows that $V_n$ intersects each vertical leaf in at most one point.
\end{proof}

\begin{prop}
Let $x \in \Omega \cap J$. Then there exists a local unstable manifold through $x$.
\end{prop}
\begin{proof}
For each $n \in \mathbb N$, the semi-local strong stable manifold $W^s_R(f^{-n}(x))$ must intersect the horizontal disk $\mathbb L_0$. Let $y_n$ be such an intersection point. Since the degree of the semi-local strong stable manifolds is uniformly bounded and $f$ is uniformly contracting on the family of strong stable manifolds, it follows that $f^n(y_n)$ converges exponentially fast to $x$. Moreover, by the exponential contraction in the vertical direction, it follows that locally the disks $D_n \subset f^n\mathbb L_0$ passing through $f^n(y_n)$ converge to a holomorphic disk through $x$, which we denote by $D$.

We claim that $D$ must be an unstable manifold, i.e. that the diameter of $f^{-n}(D)$ converges to $0$. By the exponential contraction in the vertical direction it suffices to show that the diameter of $f^{-n}(D_n)$ converges to zero. Suppose that this is not the case. Recall that the diameters of the disks $f^{-n}(D_n)$ are uniformly bounded, hence are given by images of a normal family of holomorphic maps. Hence if the diameters do not converge to zero, then there is a subsequence $(n_j)$ for which $f^{-n_j}(D_{n_j}) \subset \mathbb L_0$ converges to a holomorphic disk $E$, necessarily intersecting $J^+$. By construction $\mathbb L_0$ is transverse to $J^+$, hence it follows that $E$ must intersect vertical leaves through points in the basin of infinity, say in a point $t$. Since those vertical leaves must be contained in the basin of infinity, so must $t$. Note that since $t \in E$, it follows that $t \in f^{-n_j}(D_{n_j})$ for $j$ large enough. This however leads to a contradiction, as the forward orbit $t$ must escape the bidisk in finite time, and hence $f^{n_j}(t)$ cannot be contained in $D_{n_j}$ for $j$ large. The contradiction finishes the proof.
\end{proof}

Since $D$ is a uniform limit of holomorphic disks with horizontal tangent bundles, it follows that the tangent bundles to $D$ must also lie in the horizontal cone field. The size of the local unstable manifolds is uniform on any subset of $J$ that is bounded away from the parabolic cycles.

\begin{corollary}
The map $f$ does not have any attracting-rotating Fatou components.
\end{corollary}
\begin{proof}
Suppose there does exist an attracting rotating Fatou component, and let $U$ be the connected component of the intersection with $\Delta^2_R$ that contains the rotating disk $\Sigma$. Let $x \in \partial \Sigma$, which implies that $x \in J$.
The local unstable manifold $D$ through $x$ is horizontal and thus intersects all nearby strong stable manifolds. Since $U$ is foliated by strong stable manifolds, $D$ must intersect $U$ in a point $y$ local strong stable manifold through $\Sigma$. The sets $f^{-n}(D)$ may all be assumed to have arbitrarily small diameters, from which it follows that $f^{-n}(y)$ remains arbitrarily close to $\Sigma$. It follows that $y \in \Sigma$.  But while the backwards orbits of points on this local unstable manifold converge to $J$, backwards orbits in $\Sigma$ do not, giving a contradiction.
\end{proof}

\begin{prop}
If $f$ does not have any parabolic cycles then $f$ is hyperbolic.
\end{prop}
\begin{proof}
If $f$ does not have parabolic cycles, then there exist local unstable manifolds through any point in $J$, and the family of these unstable manifolds is invariant under $f$. By compactness of $J$ there are uniform estimates on the rate at which these unstable manifolds are contracted, hence the center direction of the dominated splitting is actually unstable, proving that the map is hyperbolic.
\end{proof}

\begin{prop}
The Julia set $J^+$ has zero Lebesgue measure.
\end{prop}
\begin{proof}
The one-dimensional counterpart of this result was proven in  \cite{Lyubich1,Orsaynotes}.
The 1D argument can be adapted to our setting as follows.
Take any point $x\in J^+$ that does not belong to strong stable manifolds of the parabolic points.
Then there is a sequence of moments  $n_k\to \infty$ for which $f^{n_k} x$ stays away from  the parabolic points.
Let $L$ be a complex horizontal line through $x$.
Proposition 10.3  implies
that there exists a shrinking nest  of ovals $V_k \subset L $
of bounded shape around $x$ (i.e., with bounded ratio of the inner and outer radii centered at $x$)
such that:

\smallskip\noindent  (i)
Each $D_k:=  f^{n_k} (V_k) $ is a horizontal-like oval of a definite size and bounded shape around $f^{n_k} x$;

\smallskip\noindent  (ii)
The maps $f^{n_k} : V_k \rightarrow D_k$ have a bounded distortion.

\smallskip
It follows from (i) that the ovals $D_k$ contain gaps in $J^+$ of definite relative size.
Then (ii)  implies that  so do the ovals $V_k$.
Hence $x$ is not a Lebesgue density point for $L\cap J^+$.

By the Lebesgue Density Points Theorem, the horizontal slices of $ J^+$ have zero area.
Since the dynamical vertical lamination of $J^+$ is smooth (Lemma 5.3),
 $J^+$ has zero volume.
\end{proof}

In the complex line the fact that a rational function has only finitely many attracting or parabolic cycles is a direct consequence of there only being finitely many critical points. For complex H\'enon maps this kind of argument cannot be used. Indeed, there do exist H\'enon maps with infinitely many attracting cycles, see for example \cite{Buzzard}. It was shown in \cite{BS1991a} that a hyperbolic H\'enon map has only finitely many Fatou components, each contained in the basin of an attracting cycle.

\begin{corollary}
There are at most finitely many attracting cycles.
\end{corollary}
\begin{proof}
This follows from the fact that an attracting cycle cannot be completely contained in the neighborhood of $J^+$ where the dominated splitting exists. But the complement of this neighborhood is contained in only finitely many Fatou components.
\end{proof}

\begin{corollary}
There are at most finitely many parabolic cycles.
\end{corollary}
\begin{proof}
By the existence of unstable manifolds through any point in $J \cap \Omega$ it follows that there can be no parabolic cycles in $\Omega$, which implies that the only parabolic cycles can be those finitely many that were removed from $J$ in the construction of $\Omega$.
\end{proof}

\begin{corollary}
The map $f$ lies on the boundary of the hyperbolicity locus in the family of H\'enon like maps.
\end{corollary}
\begin{proof}
Under perturbation of the H\'enon map the parabolic cycles bifurcate. Since we consider the infinite dimensional family of H\'enon like maps, we can work in a finite dimensional analytic parameter space whose dimension is as large as the number of parabolic cycles, and for which all parabolic cycles bifurcate for nearby parameters. For each parabolic cycle there is a suitable half-space in the parameter space which leads to stable perturbations. We consider sufficiently nearby parameters in the intersection of those half-spaces. The Julia set $J$ changes continuously in the Hausdorff dimension for such perturbations, hence the dominated splitting on $J$ is preserved. Moreover, by the fact that the horizontal direction is uniformly stable for points in $J$ bounded away from the parabolic cycles, it follows that no new parabolic cycles are constructed under sufficiently small perturbations. Thus the
\end{proof}

We would like to state that $f$ lies on the boundary of the hyperbolicity locus even in the family of polynomial H\'enon maps of the same degree. In order to follow the one-dimensional proof one needs to prove that the number of parabolic cycles is bounded by the degree, as was proved for one dimensional polynomials by Douady and Hubbard \cite{DH}, and for rational functions by Shishikura \cite{Shishikura}. We are not aware of any bound on the number of parabolic cycles for complex H\'enon maps.

\begin{corollary} The two Julia sets $J^\star$ and $J$ are equal.
\end{corollary}

In \cite{GP} the assertion $J = J^\star$ is proved using different methods, under the weaker assumption that the substantially dissipative H\'enon map $f$ admits a dominated splitting on the potentially smaller set $J^\star$.

\begin{proof}
Let $x \in J$ be a point that is not contained in one of the parabolic cycles. Let $p$ be a saddle periodic point. By \cite{BS1991b} the stable manifold $W^s(p)$ accumulates on all of $J^+$, and hence also on $x$, and similarly $W^u(p)$ must accumulate on $x$. Note that the tangent space to $W^s(p)$ must be vertical at all points.

A sufficiently small local unstable manifold $W^u_{\mathrm{loc}}(p)$ is contained in $\Omega$. Since $\Omega$ is relatively backwards invariant, it follows that if $y \in W^u_{\mathrm{loc}}(p)$ is such that $f^n(y)$ contained in $\Omega$ for some $n \in \mathbb N$, then $f^j(y) \in \Omega$ for $j = 0, \ldots n$. Since $\Omega$ is contained in the region of dominated splitting and the tangent bundle to $W^u_{\mathrm{loc}}(p)$ is horizontal, and the horizontal cone field is forward invariant, it follows that the tangent space to $W^u_{\mathrm{loc}}(f^n(p))$ at $f^n(y)$ is horizontal. Thus $W^u(p)$ is horizontal in a small neighborhood of $x$.

Since both $W^s(p)$ and $W^u(p)$ accumulate at $x$ and are respectively vertical and horizontal near $p$, it follows that there exist intersection points of $W^s(p)$ and $W^u(p)$ arbitrarily close to $x$.  By standard construction there are saddle periodic points arbitrarily close to homoclinic intersection points, and hence also arbitrarily close to $x$. Since $J^\star$ is the closure of the set of saddle points, it follows that $x \in J^\star$.

Since $J^\star$ is closed and we proved that all but finitely many points of $J$ are contained in $J^\star$, the fact that $J$ does not have isolated points implies that $J \subset J^\star$.
\end{proof}

\end{document}